\documentclass[10pt,reqno]{amsart}
\usepackage[colorlinks=true,
pdfstartview=FitV, linkcolor=blue, citecolor=green,
urlcolor=]{hyperref}
\usepackage{amsmath,amsfonts,latexsym,amssymb}
\usepackage{amssymb,amsfonts, amsmath}
\usepackage{mathrsfs}
\usepackage[latin1]{inputenc}
\usepackage[T1]{fontenc}
\usepackage{ae,aecompl}
\usepackage{braket}
\usepackage{comment}
\usepackage{amssymb,amsfonts, amsmath}
 \usepackage{mathtools} 
\usepackage[all,arc]{xy}
\usepackage{enumerate}
\usepackage{mathrsfs}
\usepackage{color}
\usepackage{subfig}
\usepackage{verbatim} 
\usepackage{amssymb,amsfonts, amsmath}
\usepackage[all,arc]{xy}
\usepackage{enumerate}
\usepackage{mathrsfs}

\usepackage{hyperref}
\usepackage{xstring}
\usepackage{amsmath}
\usepackage[makeroom]{cancel}

\newtheorem{theorem}{Theorem}[section]
\newtheorem{lemma}[theorem]{Lemma}
\newtheorem{proposition}[theorem]{Proposition}
\newtheorem{corollary}[theorem]{Corollary}
\newtheorem{definition}[theorem]{Definition}

\newtheorem{Example}[theorem]{Example}

\newtheorem{fact}[theorem]{Fact}
\newtheorem{observation}[theorem]{Observation}
\newtheorem{rmk}[theorem]{\normalfont{\em{Remark}}}
\renewcommand{\leq}{\leqslant}
\renewcommand{\geq}{\geqslant}
\long\def\@savemarbox#1#2{\global\setbox#1\vtop{\hsize\marginparwidth 
  \@parboxrestore\tiny\raggedright #2}}
\usepackage{braket}
\usepackage{amsmath}
\makeatletter
\renewcommand*\env@matrix[1][\arraystretch]{%
\edef\arraystretch{#1}%
\hskip -\arraycolsep
\let\@ifnextchar\new@ifnextchar
\array{*\c@MaxMatrixCols c}}
\makeatother
\date{\today}

\title{Cartan projections of fiber products and non quasi-isometric embeddings}
\date{\today}
\AtEndDocument{\bigskip{\footnotesize%
\textsc{Max Planck Institute for Mathematics in the Sciences, Inselstra{\ss}e 22, 04103 Leipzig, Germany} \par  
 \textit{E-mail address}: \texttt{konstantinos.tsouvalas@mis.mpg.de} }}
\makeatletter
\@namedef{subjclassname@2020}{\textup{2020} Mathematics Subject Classification}
\makeatother
\author{Konstantinos Tsouvalas}
\begin{document}
\maketitle

\begin{abstract} Let $\Gamma$ be a finitely generated group and $N$ be a normal subgroup of $\Gamma$. The fiber product of $\Gamma$ with respect to $N$ is the subgroup $\Gamma \times_N \Gamma=\big \{(\gamma, \gamma w): \gamma \in \Gamma, w \in N\big \}$ of the direct product $\Gamma \times \Gamma$. For every representation $\rho:\Gamma \times_N \Gamma \rightarrow \mathsf{GL}_d(k)$, where $k$ is a local field, we establish upper bounds for the norm of the Cartan projection of $\rho$  in terms of a fixed word length function on $\Gamma$. As an application, we exhibit examples of finitely generated and finitely presented fiber products $P=\Gamma \times_N \Gamma$, where $\Gamma$ is linear and Gromov hyperbolic, such that $P$ does not admit linear representations which are quasi-isometric embeddings.

\end{abstract}
\section{Introduction} \label{introduction}

Let $\Gamma$ be a finitely generated group and $N$ a normal subgroup of $\Gamma$. The fiber product of $\Gamma$ with respect to $N$ is the subgroup $$\Gamma \times_N \Gamma=\big\{(\gamma, \gamma w): \gamma \in \Gamma, w \in N\big\}$$ of the direct product $\Gamma \times \Gamma$. Fiber products have been used several times for the construction of pathological examples in geometric group theory, see for example \cite{Mihailova, PT, Bass-Lubotzky, Bridson-Grunewald,Bri, BR}. However, very little is known about their linear representations. In this article, we study the Cartan projection of linear representations of fiber products of the form $\Gamma\times_N \Gamma$. More precisely, for a representation \hbox{$\rho:\Gamma \times_{N}\Gamma \rightarrow \mathsf{GL}_d(k)$}, where $k$ is a local field, we establish upper bounds for the norm of the Cartan projection of $\rho$ in terms of a fixed left invariant word metric on $\Gamma$. As an application of our estimates, we exhibit examples of finitely generated and finitely presented fiber products which fail to admit linear representations which are quasi-isometric embeddings.

\par There are several important classes of discrete subgroups of semisimple Lie groups which are quasi-isometrically embedded in the ambient Lie group. These include the class of Anosov hyperbolic groups (see \cite{Labourie, GW}) and more generally of convex cocompact subgroups of the projective linear group $\mathsf{PGL}_d(\mathbb{R})$ introduced in \cite{DGK}, as well as irreducible lattices into higher rank semisimple Lie groups \cite{LMR}. It is a natural question to determine classes of linear finitely generated groups with the property that they (or do not) admit discrete faithful linear representations which are quasi-isometric embeddings. Motivated by this question, we exhibit examples of finitely generated and finitely presented fiber products $P$ such that any representation of $P$ into a general linear group, over a local field, is not a quasi-isometric embedding. Our examples will be constructed as fiber products of the form $P=\Gamma\times_N \Gamma$, where $\Gamma$ is a free group or a $C'(\frac{1}{6})$ small cancellation group of cohomological dimension $2$. To our knowledge, these are the first known examples of subgroups of direct products of Gromov hyperbolic groups with this property.
\par Before we state the main results of this paper let us provide some notation. Let $k$ be a local field. We denote by \hbox{$\mu:\mathsf{GL}_d(k) \rightarrow \mathsf{E}^{+}$} the Cartan projection on $\mathsf{GL}_d(k)$ to a cone $\mathsf{E}^{+}$ in $\mathbb{R}^d$, equipped with the standard Euclidean norm $||\cdot||_{\mathbb{E}}$. We equip a finitely generated group $\Delta$ with a left invariant word metric induced by a finite generating subset of $\Delta$ and denote by $|\cdot|_{\Delta}:\Delta \rightarrow \mathbb{N}$ the associated word length function.  A linear representation $\psi:\Delta\rightarrow \mathsf{GL}_d(k)$ is called a {\em quasi-isometric embedding} if the norm of the Cartan projection of $\psi(\Delta)$ grows uniformly linearly in the word length on $\Delta$. In other words, there exists $C>1$ such that for every $\gamma \in \Delta$, $$C^{-1}\big|\gamma\big|_{\Delta}-C\leq \Big| \Big| \mu\big(\psi(\gamma)\big)\Big|\Big|_{\mathbb{E}} \leq C\big|\gamma\big|_{\Delta}+C.$$  
 
For a group $\mathsf{H}$ and $w_1,w_2 \in \mathsf{H}$, define the commutator $[w_1,w_2]:=w_1^{-1}w_2^{-1}w_1w_2$ and inductively set $[w_1,\ldots, w_{r-1}, w_r]:=\big[[w_1, \ldots, w_{r-1}], w_r\big]$ for every $w_1,\ldots, w_{r-1},w_r \in \mathsf{H}$. The commutator subgroup of $\mathsf{H}$ is $[\mathsf{H},\mathsf{H}]=\big \langle \{[w,w']:w,w'\in \mathsf{H}\}\big \rangle$. Our first result concerns linear representations of a direct product of the form $\mathsf{H}\times \mathsf{H}$. We establish an upper bound for the norm of the Cartan projection of multiple commutators in $\{1\}\times \mathsf{H}$, in terms of the norm of the Cartan projection of the diagonal subgroup $\textup{diag}(\mathsf{H} \times \mathsf{H})=\big\{(w,w):w\in \mathsf{H}\big\}$.

\begin{theorem} \label{mainthm2} Let $k$ be a local field and $\mathsf{H}$ be a group. For every representation \hbox{$\rho:\mathsf{H} \times \mathsf{H} \rightarrow \mathsf{GL}_d(k)$} there exists $C>1$, depending only on $\rho$, with the property: for every $w_1,w_2,\ldots ,w_{r} \in \mathsf{H}\smallsetminus\{1\}$ and $r \geq d+1$ we have that $$\Big| \Big| \mu \big(\rho \big(1, [w_1,w_2,\ldots, w_{r}]\big) \big)\Big|\Big|_{\mathbb{E}} \leq 2^r C\Big(1+\sum_{i=1}^{r}\Big|\Big|\mu\big(\rho(w_i,w_i)\big)\Big|\Big|_{\mathbb{E}}\Big).$$ \end{theorem}
\medskip

As a corollary of Theorem \ref{mainthm2} we obtain the following bound for the Cartan projection of  linear representations of (not necessarily finitely generated) fiber products.

\begin{corollary} \label{mainthm2-fiberproducts} Let $\Gamma$ be a finitely generated group, $N<\Gamma$ a normal subgroup and \hbox{$|\cdot|_{\Gamma}:\Gamma \rightarrow \mathbb{N}$} a word length function on $\Gamma$. Let $k$ be a local field. For every representation $\rho:\Gamma \times _N \Gamma \rightarrow \mathsf{GL}_d(k)$ there exist $C,c>1$, depending only on $\rho$, with the property: for every $w_1,w_2,\ldots ,w_{r} \in N$, $r \geq d+1$, and $\gamma \in \Gamma$ we have that $$\Big| \Big| \mu \big(\rho \big(\gamma, \gamma [w_1,w_2,\ldots, w_{r}]\big) \big)\Big|\Big|_{\mathbb{E}} \leq 2^r C\Big(\sum_{i=1}^{r}\big|w_i\big|_{\Gamma}\Big)+c\big|\gamma \big|_{\Gamma}.$$ \end{corollary}

Let $F_m$ be the free group on $m\geq 2$ generators. As an application of Theorem \ref{mainthm2} and Corollary \ref{mainthm2-fiberproducts}, we exhibit examples of finitely generated subgroups of the direct product $F_m \times F_m$ which fail to admit linear representations which are quasi-isometric embeddings. 

\begin{theorem} \label{nonqie1} Let $F_m$, $m \geq 2$, be the free group on $\{x_1,\ldots ,x_m\}$ and $\mathcal{A}$ a finite subset of $[F_m,F_m]$ which contains the commutator $[x_{j},x_{p}]=x_{j}^{-1}x_p^{-1}x_{j} x_p$ for some $1\leq p<j \leq m$. Let $$\Delta_{\mathcal{A}}=F_m\times_{\langle \langle \mathcal{A}\rangle \rangle}F_m=\Big \langle \big \{ (x_i,x_i), (1,w): w \in \mathcal{A},\ 1 \leq i \leq m  \big \} \Big \rangle$$ be the fiber product of $F_m$ with respect to the normal subgroup \hbox{$\langle \langle \mathcal{A} \rangle \rangle=\big\langle \{ghg^{-1}:h\in \mathcal{A},g \in F_m \}\big \rangle$} and fix a word length function $|\cdot|_{\mathcal{A}}:\Delta_{\mathcal{A}} \rightarrow \mathbb{N}$. For every $d \in \mathbb{N}$, there exists an infinite sequence $(w_n)_{n \in \mathbb{N}}$ of elements in $\Delta_{\mathcal{A}}$ with the property: for every representation \hbox{$\rho:\Delta_{\mathcal{A}} \rightarrow \mathsf{GL}_d(k)$}, where $k$ is a local field, there exists $C_{\rho}>0$ such that \begin{equation} \label{intro-ineq1} \Big|\Big| \mu \big(\rho(w_n)\big)\Big|\Big|_{\mathbb{E}} \leq C_{\rho}\big|w_n\big|_{\mathcal{A}}^{1-\frac{1}{2d+3}}\end{equation} for every $n \in \mathbb{N}$. In particular, the finitely generated group $\Delta_{\mathcal{A}}$ does not admit linear representation, over a local field, which is a quasi-isometric embedding. \end{theorem}

Now we briefly explain how Theorem \ref{mainthm2} and Corollary \ref{mainthm2-fiberproducts} are used in the proof of Theorem \ref{nonqie1}. Theorem \ref{mainthm2} and Corollary \ref{mainthm2-fiberproducts} imply that for any representation $\rho: F_m\times_{\langle \langle \mathcal{A} \rangle \rangle} F_m \rightarrow \mathsf{GL}_d(k)$, $d\in \mathbb{N}$, and any sequence of the form $(1,y_n)\in F_m\times_{\langle \langle \mathcal{A} \rangle \rangle} F_m$, where $y_n\in [F_m,F_m]$ is a commutator of at least $d+1$ elements in $\langle \langle \mathcal{A}\rangle \rangle$, the norm of the Cartan projection of the sequence $\big(\rho(1,y_n)\big)_{n\in \mathbb{N}}$ grows at most linearly in $n$. However, for particular choices of $(y_n)_{n\in \mathbb{N}}$, the word length of the sequence $((1,y_n))_{n\in \mathbb{N}}$ grows faster than linearly in $n$, showing that $\rho$ fails to be a quasi-isometric embedding. One possible choice of such a sequence is $$w_n:=\Big(1,\Big[\big[ [x_j^n,x_p^n],x_p^n\big],\underbrace{[x_j^n,x_p^n], \ldots, [x_j^n,x_p^n]}_{d-\textup{times}}\Big]\Big).$$ We refer to Theorem \ref{nonqie1gen} which is a generalization of the previous theorem for more details.

Thanks to the Rips construction \cite{Rips} and its generalizations (e.g. see \cite{Rips-Wise}), given any finitely presented group $Q$ there exist explicit examples of a Gromov hyperbolic group $\Gamma$, satisfying the $C'(\frac{1}{6})$ small cancellation condition, and a normal finitely generated subgroup $N$ of $\Gamma$ such that $\Gamma/N=Q$. By using Corollary \ref{mainthm2-fiberproducts} and following the point of view of the Bridson--Grunewald construction in \cite{Bridson-Grunewald}, we obtain the following theorem.

\begin{theorem} \label{nonqie2} Let $F_m$, $m \geq 2$, be the free group on $\{x_1,\ldots ,x_m\}$ and $\mathcal{B}$ a finite subset of $[F_m,F_m]$ which contains the commutator $[x_{j},x_{p}]=x_{j}^{-1}x_p^{-1}x_{j} x_p$ for some \hbox{$1\leq p<j \leq m$.} Let $Q_{\mathcal{B}}=\big \langle x_1,\ldots,x_m\ \big| \ \mathcal{B} \big\rangle$ and $\Gamma_{\mathcal{B}}$ be a $C'(\frac{1}{6})$ small cancellation group provided by the Rips construction\footnote{We refer here to Rips' construction in \cite{Rips} which we review in \S\ref{Rips-construction}.} such that $$1 \rightarrow N_{\mathcal{B}} \rightarrow \Gamma_{\mathcal{B}} \rightarrow Q_{\mathcal{B}} \rightarrow 1$$ is a short exact sequence and $N_{\mathcal{B}}$ is finitely generated. Let $P_{\mathcal{B}}= \Gamma_{\mathcal{B}} \times_{N_{\mathcal{B}}} \Gamma_{\mathcal{B}}$ and fix a word length function $|\cdot|_{\mathcal{B}}:P_{\mathcal{B}} \rightarrow \mathbb{N}$. For every $d \in \mathbb{N}$, there exists an infinite sequence $(\delta_{n})_{n \in \mathbb{N}}$ of elements in $P_{\mathcal{B}}$ with the property: for every representation $\rho:P_{\mathcal{B}} \rightarrow \mathsf{GL}_d(k)$, where $k$ is a local field, there exists $C_{\rho}>0$ such that \begin{equation} \label{intro-ineq2} \Big|\Big| \mu \big(\rho(\delta_n)\big)\Big|\Big|_{\mathbb{E}} \leq C_{\rho}\big|\delta_n\big|_{\mathcal{B}}^{1-\frac{1}{2d+3}}\end{equation} for every $n \in \mathbb{N}$. In particular, the finitely generated group $P_{\mathcal{B}}$ does not linear representation, over a local field $k$, which is a quasi-isometric embedding.\end{theorem} 

\noindent For explicit examples of sequences $(\delta_n)_{n\in \mathbb{N}}$, depending only on $d\in \mathbb{N}$ and the finite set $\mathcal{B}$, satisfying the upper bound (\ref{intro-ineq2}), we refer to \hbox{Theorem \ref{nonqie2gen}.}

\par Let us also remark that when $Q_{\mathcal{B}}$ is of type $\mathcal{F}_3$ (i.e. admits a $K(Q_{\mathcal{B}},1)$ Eilenberg--MacLane space whose $3$-skeleton is finite) the fiber product $P_{\mathcal{B}}$ in Theorem \ref{nonqie2} is finitely presented by the 1-2-3-Theorem in \cite{BBMS}. Moreover, it follows by the work of Agol \cite[Cor. 1.2]{Agol} and Wise \cite[Thm. 1.2]{Wise} that the Gromov hyperbolic group $\Gamma$ is virtually special and hence admits a discrete faithful representation into $\mathsf{GL}_{d}(\mathbb{C})$ for some $d \in \mathbb{N}$. In fact, as a consequence of \cite[Thm. 1.3]{DFWZ}, $\Gamma$ admits a complex representation which is a quasi-isometric embedding. Therefore, in the view of these facts and Theorem \ref{nonqie2}, we obtain the following corollary.

\begin{corollary}\label{fp1} Let $\mathcal{B}, N_{\mathcal{B}}, \Gamma_{\mathcal{B}}, Q_{\mathcal{B}}$ and $P_{\mathcal{B}}=\Gamma_{\mathcal{B}}\times_{N_{\mathcal{B}}}\Gamma_{\mathcal{B}}$ be as in Theorem  \ref{nonqie2}. Suppose that the finitely presented group $Q_{\mathcal{B}}$ is of type $\mathcal{F}_3$ \textup{(}e.g. free abelian\textup{)}. Then $P_{\mathcal{B}}$ is a finitely presented subgroup of $\Gamma_{\mathcal{B}} \times \Gamma_{\mathcal{B}}$ which admits a discrete faithful complex representation but does not admit linear representation, over a local field, which is a quasi-isometric embedding.\end{corollary}

We provide an explicit example and its finite presentation with $6$ generators and $21$ relations in Example \ref{exmp-fp}. To our knowledge the groups provided by Corollary \ref{fp1} are the first known examples of finitely presented subgroups of CAT(0) groups which fail to admit linear representations which are quasi-isometric embeddings.

\subsection*{More examples from Grothendieck pairs.} In \S\ref{Gr-pairs} we provide some other examples of fiber products, which are not commensurable to those of Theorem \ref{nonqie1} and Theorem \ref{nonqie2}, and fail to admit admit any linear representation which is a quasi-isometric embedding. This is achieved thanks to Theorem \ref{PTGr} which is a stronger version of Corollary \ref{mainthm2-fiberproducts} for fiber products of the form $\Gamma\times_N \Gamma$, where $\Gamma/N$ satisfies the conditions of the Platonov--Tavgen criterion \cite{PT} (i.e. $\Gamma/N$ does not admit non-trivial finite quotients and $H_2(\Gamma/N,\mathbb{Z})=0$). The main tool for the proof of Theorem \ref{PTGr} is a theorem of Grothendieck \cite[Thm. 1.2]{Grothendieck} which establishes a connection between the profinite completion of a finitely generated group and its representation theory.

\medskip
\noindent {\bf Organization of the paper.} In \S\ref{prelim} we provide some necessary background on fiber products, the Cartan projection, proximality and prove some preliminary lemmas that we will use in the following sections. In \S\ref{proofs1} we are going to prove Theorem \ref{mainthm2} and in \S\ref{Lemmas} we establish some further necessary lemmas. In \S\ref{Rips-construction} we recall Rips' construction from \cite{Rips} and in \S\ref{proofs2} we prove Theorem \ref{nonqie1} and Theorem \ref{nonqie2}. In \S\ref{Examples} we provide an explicit example of a finite presentation of a fiber product satisfying the conclusion of Theorem \ref{nonqie2}. In \S\ref{Gr-pairs} we  discuss some known constructions of Grothendieck pairs and further examples of fiber products failing to admit representations which are quasi-isometric embeddings.\\

\noindent \textbf{Acknowledgements.} I would like to thank Richard Canary, Fanny Kassel, Claudio Llosa Isenrich, Alan Reid and Mihalis Sykiotis for interesting discussions, as well as the anonymous referees for their comments. I would also like to thank IH\'ES, where this project started, for providing excellent working conditions. This project received funding from the European Research Council (ERC) under the European's Union Horizon 2020 research and innovation programme (ERC starting grant DiGGeS, grant agreement No 715982).

\section{Preliminaries} \label{prelim}
In this section, we provide some necessary background and prove certain lemmas that we are going to use in the following sections of the paper.

\par Let $\Gamma$ be a group. For a subset $\mathcal{F}$ of $\Gamma$, the {\em normal closure of $\mathcal{F}$ in $\Gamma$} is the normal subgroup $\langle \langle \mathcal{F} \rangle \rangle=\big \langle \gamma f \gamma^{-1}:\gamma \in \Gamma, f \in \mathcal{F} \big \rangle$. If $\Gamma$ is a finitely generated group we equip it with a left invariant word metric $d_{\Gamma}:\Gamma \times \Gamma \rightarrow \mathbb{N}$ induced by some finite generating subset $S$ of $\Gamma$. We denote by $|\cdot|_{\Gamma}:\Gamma \rightarrow \mathbb{N}$, $|\gamma|_{\Gamma}=d_{\Gamma}(\gamma,1)$, $\gamma \in \Gamma$, the associated word length function. The group $\Gamma$ is called {\em Gromov hyperbolic} if its Cayley graph with respect to $S$, equipped with $d_{\Gamma}$, is a $\delta$-hyperbolic space. We refer the reader to \cite{Gromov}, \cite[part III]{BH} and \cite{CDP} for more background on Gromov hyperbolic spaces.\par Let us recall once more that if $N$ is a normal subgroup of $\Gamma$, the {\em fiber product of $\Gamma$ with respect to $N$} is the subgroup of $\Gamma \times \Gamma$ generated by the diagonal and the subgroup $\{1\} \times N$: $$\Gamma \times_N\Gamma=\big\{(\gamma,\gamma w) \in \Gamma \times \Gamma:\gamma \in \Gamma, \ w \in N \big\}.$$ 
\begin{fact} If $N=\langle \langle \mathcal{F} \rangle \rangle$, where $\mathcal{F}$ is a finite subset of $\Gamma$, then $\big \{(g,g):g \in S\big\} \cup \big \{(1,h):h \in \mathcal{F}\big\}$ is a finite generating subset of $\Gamma \times_N \Gamma$.\end{fact}

\subsection{Proximality} We denote by $k$ a local field, i.e $\mathbb{R}$, $\mathbb{C}$ or a finite extension of $\mathbb{Q}_p$ or the field of formal Laurent series $\mathbb{F}_q((t))$ over the finite field $\mathbb{F}_q$ and by $|\cdot|:k \rightarrow \mathbb{R}^{+}$ the absolute value on $k$. If $k$ is Archimedean (i.e. $k=\mathbb{R}$ or $\mathbb{C}$) $|\cdot|$ is the standard Euclidean absolute value. If $k$ is non-Archimedean and $\omega:k \rightarrow \mathbb{Z}\cup \{\infty\}$ is a discrete valuation on $k$, the absolute value is $|\cdot|=q^{-\omega(\cdot)}$, where $q\in \mathbb{N}$ is the cardinality of the residue field of $k$. The ring of integers of $k$ is $\mathcal{O}:=\big \{x \in K:|x|\leq 1\big \}$ and we fix $\pi \in \mathcal{O}$ a uniformizer with $\omega(\pi)=1$.

\par The algebra of $d\times d$ matrices with entries in $k$ is denoted by $\mathfrak{gl}_d(k)$. The $\ell_2$-matrix norm \hbox{$||\cdot||:\mathfrak{gl}_d(k) \rightarrow \mathbb{R}^{+}$} is defined as follows $$\big|\big|B\big|\big|=\Big({\sum_{i,j=1}^d\big|b_{ij}\big|^2}\Big)^{\frac{1}{2}}$$ where $B=(b_{ij})_{i,j =1}^{d}, b_{ij}\in k$. For $1 \leq i,j \leq d$, $E_{ij}$ denotes the elementary matrix with $1$ at the $(i, j)$-entry and $0$ everywhere else.
\par  Let $g \in \mathsf{GL}_d(k)$ be a matrix. The eigenvalues of $g$ lie in a finite extension $k'$ of $k$. We denote by $\ell_1(g)\geq \ldots \geq \ell_d(g)$ the absolute values of the eigenvalues of $g$ in non-increasing order. Note that for every $1 \leq i \leq d-1$ we have $\ell_{i}(g^{-1})=\ell_{d-i+1}(g)^{-1}$. An element $g \in \mathsf{GL}_d(k)$ is called {\em 1-proximal} if $\ell_1(g)>\ell_2(g)$. If $g$ is $1$-proximal then it has a unique eigenvalue of maximum absolute value, necessarily in $k$. In this case, $g\in \mathsf{GL}_d(k)$ admits a unique attracting fixed point $x_g^{+}\in \mathbb{P}(k^d)$ and a repelling $(d-1)$-hypeplane $V_{g}^{-}$ such that $k^d=x_{g}^{+}\oplus V_{g}^{-}$ and for every $y \in \mathbb{P}(k^d)\smallsetminus \mathbb{P}(V_{g}^{-})$, $\lim_{n}g^ny=x_{g}^{+}$.

\subsection{Cartan and Lyapunov projection on $\mathsf{G}=\mathsf{GL}_{d}(k)$.} Let us consider the cone of $\mathbb{R}^d$ $$\mathsf{E}^{+}:=\big\{(x_1,\ldots,x_d) \in \mathbb{R}^d:x_1 \geq \ldots \geq x_d\big\}.$$
\noindent (a) {\em $k$ is Archimedean}. For a matrix $g=(g_{ij})_{i,j=1}^{d}$ denote by $g^{\ast}=(\overline{g_{ji}})_{i,j=1}^d$ and $g^t=(g_{ji})_{i,j=1}^{d}$ the conjugate transpose and transpose of $g$ respectively. If $k=\mathbb{R}$ (resp. $k=\mathbb{C}$) let $K=\mathsf{O}(d)$ (resp. $K=\mathsf{U}(d)$) be the (unique up to conjugation) maximal compact subgroup of $\mathsf{G}$. A maximal $k$-split torus of $\mathsf{G}$ is the set of diagonal matrices with entries in $k$ with dominant Weyl chamber $$A^{+}=\big \{\textup{diag}\big(a_1,\ldots, a_d\big): a_1 \geq \ldots \geq a_d, \ a_i \in \mathbb{R}^{+}\big\}.$$ The corresponding Cartan decomposition is $\mathsf{G}=KA^{+}K$, i.e. every $g \in \mathsf{G}$ is written in the form $g=w_ga_gw_g'$, where $w_g,w_g' \in K$ and $a_g \in A^{+}$. The {\em Cartan projection} is the map $\mu:\mathsf{G}\rightarrow \mathsf{E}^{+}$ defined as follows $$\mu(g)=\big(\log\sigma_1(g),\ldots,\log \sigma_d(g) \big)$$ where $\sigma_i(g)=\sqrt{\ell_i(gg^{\ast})}$ is the {\em $i$-th singular value of $g$}. The map $\mu$ is a continuous, proper and surjective map. 
\medskip

\noindent (b) {\em $k$ is non-Archimedean.} Let $p\in \mathbb{N}$ be the cardinality of the residue field of $k$. The group \hbox{$K=\mathsf{GL}_d(\mathcal{O})$} is the unique, up to conjugation, maximal compact subgroup of $\mathsf{G}$. A maximal $k$-split torus of $\mathsf{G}$ is the set of diagonal matrices with entries in $k$ with dominant Weyl chamber $$A^{+}=\big \{\textup{diag}\big(a_1,\ldots, a_d\big): |a_1| \geq \ldots \geq |a_d|,\ a_i \in k^{\ast} \big\}$$ and the corresponding Cartan decomposition is $\mathsf{G}=KA^{+}K$. Every matrix $g \in \mathsf{G}$ can be written in the form $g=w_ga_g^{+}w_{g}'$, where $w_g,w_g' \in K$ and $$a_g^{+} =\textup{diag}\big(\pi^{n_1(g)},\ldots, \pi^{n_d(g)}\big)$$ where $n_1(g),\ldots,n_d(g)\in \mathbb{Z}$ and $n_d(g) \geq \ldots \geq n_1(g)$. For $1 \leq i \leq d$, the \hbox{{\em $i$-th singular value of $g$} is} $$\sigma_i(g)=\big|\pi^{n_i(g)}\big|=p^{-n_i(g)}.$$ The Cartan projection $\mu:\mathsf{G} \rightarrow \mathsf{E}^{+}$ is the vector valued map $$\mu(g)=\big(\log \sigma_1(g),\ldots \log \sigma_d(g) \big), \ g \in \mathsf{G}.$$ In this case, $\mu$ is  proper and continuous map but not surjective onto $\mathsf{E}^{+}$.
\par In both cases where $k$ is Archimedean or not, the {\em Lyapunov projection} is the map $\lambda: \mathsf{G}\rightarrow \mathsf{E}^{+}$ defined in terms of the Cartan projection as follows $$\lambda(g)=\lim_{r \rightarrow \infty}\frac{1}{r}\mu(g^r)=\big(\log \ell_1(g),\ldots \log \ell_d(g) \big), \ g \in \mathsf{G}.$$

Let us also recall the definition of a linear representation, over a local field, being a quasi-isometric embedding.

\begin{definition} Let $\mathsf{H}$ be a finitely generated group and fix $|\cdot|_{\mathsf{H}}:\mathsf{H} \rightarrow \mathbb{N}$ a word length function on $\mathsf{H}$. A representation $\rho:\mathsf{H} \rightarrow \mathsf{GL}_d(k)$ is called a quasi-isometric embedding if there exists $C>1$ such that for every $\gamma \in \mathsf{H}$ we have $$C^{-1}\big|\gamma \big|_{\mathsf{H}}-C \leq \Big| \Big| \mu \big(\rho(\gamma)\big)\Big|\Big|_{\mathbb{E}} \leq C\big|\gamma \big|_{\mathsf{H}}+C.$$ \end{definition}

\begin{rmk}\normalfont{If $k=\mathbb{R}$ or $\mathbb{C}$, let $(X,\mathsf{d})$ be the Riemannian symmetric space $\mathsf{G}/K$ equipped with the Killing metric $\mathsf{d}$. If $k$ is non-Archimedean, $(X,\mathsf{d})$ denotes the Bruhat--Tits building on which $\mathsf{G}$ acts properly by isometries, see \cite{BT} for more details. A representation $\rho:\mathsf{H}\rightarrow \mathsf{GL}_d(k)$ is a quasi-isometric embedding if and only if the orbit map $\tau_{\rho}:(\mathsf{H},d_{\mathsf{H}})\rightarrow (X,\mathsf{d})$, $\tau_{\rho}(\gamma)=\rho(\gamma)x_0, \gamma \in \mathsf{H}$, is a quasi-isometric embedding between metric spaces.}\end{rmk}

\subsection{Semisimple representations} Let $\mathsf{\Gamma}$ be a group. A representation $\rho:\mathsf{\Gamma} \rightarrow \mathsf{GL}_d(k)$ is called {\em semisimple} if $\rho$ decomposes as a direct sum of irreducible representations.

\begin{fact} \label{normal} Let $\mathsf{\Gamma}$ be a group and $\mathsf{\Gamma}'\subset \mathsf{\Gamma}$ be a normal subgroup. Suppose that $\rho:\mathsf{\Gamma} \rightarrow \mathsf{GL}_d(k)$ is a semisimple representation. Then the restriction $\rho|_{\mathsf{\Gamma}'}:\mathsf{\Gamma}'\rightarrow \mathsf{GL}_d(k)$ is semisimple.\footnote{If $\rho$ is irreducible and $V'\subset k^d$ is a $\rho(\mathsf{\Gamma}')$-invariant $k$-subspace of minimum positive dimension, then, by using induction, there exist $\gamma_1,\ldots,\gamma_r\in \mathsf{\Gamma}$ such that $k^d=\bigoplus_{i=1}^{r}\rho(\gamma_i)V'$.}\end{fact}

\par The following result, established by Benoist in \cite{benoist-limitcone}, using a result of Abels--Margulis--Soifer from \cite{AMS}, offers a connection between eigenvalues and singular values of elements in the image of a semisimple linear representation over a local field. For a proof of the following result, in the case where $k$ is Archimedean, we refer the reader to \cite[Thm. 4.12]{GGKW}.

\begin{theorem}\textup{\big(\textup{Abels--Margulis--Soifer} \cite{AMS}, \textup{Benoist} \cite{benoist-limitcone}\big)} \label{finitesubset} Let $k$ be a local field. Suppose that $\mathsf{\Gamma}$ is a group and \hbox{$\rho:\mathsf{\Gamma} \rightarrow \mathsf{GL}_{d}(k)$} is a semisimple representation. There exist a finite subset $F$ of $\mathsf{\Gamma}$ and $C_{\rho}>0$ with the property: for every $g \in \mathsf{\Gamma}$ there exists $f \in F$ such that $$\Big|\Big| \mu\big(\rho(g)\big)-\lambda\big(\rho(g f)\big)\Big|\Big|_{\mathbb{E}}\leq C_{\rho}.$$ \end{theorem}

We will also need the following lemma.

\begin{lemma} \label{restriction} Let $k$ be a local field and $\mathsf{H}$ be a group. Suppose that $\psi:\mathsf{H} \times \mathsf{H} \rightarrow \mathsf{GL}_m(k)$ is a semisimple representation such that $\psi(\{1\}\times \mathsf{H})$ contains a $1$-proximal element. Then there exist finitely many group homomorphisms $\varepsilon_1,\ldots,\varepsilon_{p}:\mathsf{H} \rightarrow k^{\ast}$ with the property: if $g\in \mathsf{H}$, $\psi(1,g)$ is $1$-proximal and $v_{g}^{+}\in k^d$ is an eigenvector of $\psi(1,g)$ with respect to its eigenvalue of maximum modulus, then there exists $1 \leq r \leq p$ such that for every $h\in \mathsf{H}$, $$\psi(h,1)v_{g}^{+}=\varepsilon_r(h)v_{g}^{+}.$$ \end{lemma}

\begin{proof} Observe that $\{1\} \times \mathsf{H}$ is a normal subgroup of $\mathsf{H} \times \mathsf{H}$ and hence Fact \ref{normal} implies that $\psi|_{\{1\}\times \mathsf{H}}$ is a semisimple representation since $\psi$ is. Let $V$ be the vector $k$-subspace of $k^m$ spanned by the attracting eigenlines of $1$-proximal elements of $\psi(\{1\}\times \mathsf{H})$. Since $\psi|_{\{1\}\times \mathsf{H}}$ is semisimple, there exists a decomposition $$V=V_{1}\oplus \cdot \cdot \cdot \oplus V_{q}$$ into $\psi(\{1\}\times \mathsf{H})$-invariant subspaces such that the restriction of \hbox{$\psi|_{V_i}:\{1\}\times \mathsf{H} \rightarrow \mathsf{GL}(V_i)$} is irreducible and its image contains a $1$-proximal element. Note that for each $i$, the subspace $V_{i}$ contains a basis $\mathcal{B}_i\subset V_i$, consisting entirely of attracting eigenvectors of $1$-proximal elements of $\psi(\{1\}\times \mathsf{H})|_{V_{i}}$. Since $\psi(\mathsf{H}\times \{1\})$ centralizes $\psi(\{1\} \times \mathsf{H})$, $\psi(\mathsf{H}\times \{1\})$ fixes every vector in $\mathcal{B}_i$\footnote{Note that if $g,h\in \mathsf{GL}_d(k)$ are commuting and $g$ is $1$-proximal, then $hx_{g}^{+}, x_{g}^{+}\in \mathbb{P}(k^d)$ are attracting eigenvectors for $g$ and hence $hx_{g}^{+}=x_{g}^{+}$.} for every $1 \leq i \leq q$. \par Now let us fix $h \in \mathsf{H}$. Up to conjugation by an element of $\mathsf{GL}(V_i)$, the restriction $\psi(h,1)|_{V_i}$ on $V_i$ is a diagonal matrix of the form $$\textup{diag}\big(q_{i1}\textup{I}_{d_{1}},\ldots,q_{is}\textup{I}_{d_{s}}\big)$$ where $q_{ij}\in k^{\ast}$ and $\sum_{j=1}^{s}d_{j}=\textup{dim}_k(V_i)$, commuting with the irreducible subgroup $\psi(\{1\}\times \mathsf{H})|_{V_i}$ of $\mathsf{GL}(V_i)$. Now suppose that there exists $j\in \{1,\ldots, s\}$ such that $q_{ij} \neq q_{i\ell}$ for every $\ell \neq j$. Up to conjugation by an element of $\mathsf{GL}(V_i)$, we may assume that $j=1$. Since $\psi(h,1)$ centralizes $\psi(\{1\} \times \mathsf{H})|_{V_i}$, we check that for every $w \in \mathsf{H}$ the first row of $\psi(1,w)$ consists entirely of zeros except in the $(1,1)$ entry. It follows that $\psi(\{1\}\times \mathsf{H})|_{V_i}$ preserves a proper subspace of $V_i$ which is a contradiction. Therefore, $q_{i1}=\cdots =q_{is}$ and $s=1$. In conclusion, $\psi(h,1)$ acts as scalar multiple of the identity on $V_{i}$ for every $h \in \mathsf{H}$ and $1 \leq i \leq q$. 
\end{proof}

Denote by $\mathfrak{gl}_d(k)$ the algebra of $d\times d$ matrices with entries in $k$. Let $\Delta$ be a group and $\psi:\Delta \rightarrow \mathsf{GL}_d(k)$ be a representation. The representation $\psi$ is called {\em spanning} if the vector subspace of $\mathfrak{gl}_d(k)$ spanned by $\psi(\Delta)$ is $\mathfrak{gl}_d(k)$. We close this section with the following analogue of Goursat's lemma describing the subspace generated by a product of two spanning representations.

\begin{lemma} \label{subalgebra} Let $k$ be a local field and $\Delta$ be a group. Suppose that $\rho_1:\Delta \rightarrow \mathsf{GL}_m(k)$ and \hbox{$\rho_2:\Delta \rightarrow \mathsf{GL}_n(k)$} are two spanning representations. Let $\mathfrak{g}\subset \mathfrak{gl}_{m}(k) \times \mathfrak{gl}_{n}(k)$ be the $k$-vector subspace spanned by $(\rho_1\times \rho_2)(\Delta):=\big\{\big(\rho_1(\gamma),\rho_2(\gamma)\big):\gamma  \in \Delta \big\}$. Then one of the following holds:\\
\noindent \textup{(i)} $\mathfrak{g}=\mathfrak{gl}_{m}(k) \times \mathfrak{gl}_{n}(k)$.\\
\noindent \textup{(ii)} $n=m$ and there exists $h \in \mathsf{GL}_n(k)$ such that $\rho_2(\gamma)=h\rho_1(\gamma)h^{-1}$ for every $\gamma \in \Delta$.\end{lemma}

\begin{proof} We first observe that since $\Delta$ is a group, $\mathfrak{g}$ is a sub-algebra of $\mathfrak{gl}_m(k) \times \mathfrak{gl}_n(k)$. Let $\textup{pr}_1: \mathfrak{g} \rightarrow \mathfrak{gl}_{m}(k)$ and $\textup{pr}_2:\mathfrak{g} \rightarrow \mathfrak{gl}_{n}(k)$ be the projections on the first and second coordinate respectively. Let us assume that $\mathfrak{g}$ is a proper sub-algebra of $\mathfrak{gl}_{m}(k) \times \mathfrak{gl}_{n}(k)$. We are going to prove that $n=m$ and $\rho_1$, $\rho_2$ are conjugate. \par By assumption we have $\textup{pr}_1(\mathfrak{g})=\mathfrak{gl}_m(k)$ and $\textup{pr}_2(\mathfrak{g})=\mathfrak{gl}_{n}(k)$. Now we claim that the projection $\textup{pr}_1$ is injective. Suppose not. Then there exists $w \in \mathfrak{gl}_n(k)\smallsetminus \{0\}$ such that $(0,w)\in \mathfrak{g}$. Note that since $\mathfrak{g}$ is $(\rho_1\times \rho_2)(\Delta)$ invariant, $(0,\rho_2(\delta)w\rho_2(\gamma)) \in \mathfrak{g}$ for every $\gamma, \delta \in \Delta$. Since $\rho_2(\Delta)$ spans $\mathfrak{gl}_n(k)$, $\mathfrak{g}$ also contains the set $\mathcal{C}_w:=\big \{(0,w_1ww_2): w_1,w_2 \in \mathfrak{gl}_n(k)\big\}$. Moreover, note that there is no proper subspace of $\mathfrak{gl}_n(k)$ which contains the set $\mathcal{C}_w$, since the latter contains all the elementary $n\times n$ matrices $\{E_{ij}\}_{i,j=1}^{n}$. It follows that $\{0\}\times \mathfrak{gl}_n(k)$ is a subalgebra of $\mathfrak{g}$. In particular, $\big(\mathfrak{gl}_n(k) \times \{0\}\big) \cap \mathfrak{g}$ contains $\rho_1(\Delta)\times \{0\}$. Since $\rho_1$ is spanning, we similarly conclude that $\mathfrak{gl}_m(k)\times \{0\}$ is contained in $\mathfrak{g}$. Finally, we have $\mathfrak{g}=\mathfrak{gl}_m(k)\times \mathfrak{gl}_n(k)$ contradicting our assumption that $\mathfrak{g}$ is proper. Therefore, $\textup{pr}_1:\mathfrak{g}\rightarrow \mathfrak{gl}_m(k)$ is an algebra isomorphism. Similarly, we check that  $\textup{pr}_2:\mathfrak{g}\rightarrow \mathfrak{gl}_n(k)$ is an algebra isomorphism. \par It follows that $n=m$. By the Skolem--Noether theorem (e.g. see \cite[p.174]{Lor}) the automorphism $\textup{pr}_2\circ\textup{pr}_1^{-1}: \mathfrak{gl}_n(k) \rightarrow \mathfrak{gl}_n(k)$ is inner, i.e. there exists $h \in \mathsf{GL}_n(k)$ such that $\big(\textup{pr}_2 \circ \textup{pr}_{1}^{-1}\big)(g)=hgh^{-1}$ for every $g \in \mathfrak{gl}_n(k)$. In particular, since $\textup{pr}_{1}^{-1}(\rho_1(\gamma))=(\rho_1(\gamma),\rho_2(\gamma))$ for $\gamma \in \Delta$, we conclude that $\rho_2(\gamma)=h\rho_1(\gamma)h^{-1}$ for every $\gamma \in \Delta$.\end{proof}

\section{Proof of Theorem \ref{mainthm2}} \label{proofs1}
In this section we prove Theorem \ref{mainthm2}. In order to simplify our notation in the following statements, we introduce the following definition.

\begin{definition} Let $k$ be a local field and $\mathsf{H}$ be a group. A representation $\rho:\mathsf{H}\times \mathsf{H} \rightarrow \mathsf{GL}_{d}(k)$ is called $r$-normal if there exist $C,a>0$ such that \begin{equation} \label{reg} \Big|\Big| \rho \big(1, [w_1, w_2,\ldots, w_r]^{\pm 1} \big)\Big|\Big| \leq Ce^{ a\sum_{i=1}^{r}||\mu(\rho(w_i,w_i))||_{\mathbb{E}}}\end{equation} for every $w_1,\ldots,w_r \in \mathsf{H}$.\end{definition}

We will need the following lemma providing a sharper estimate from that of Theorem \ref{mainthm2} in the case where the representation $\rho$ is semisimple.

\begin{lemma} \label{mainthm-semisimple0} Let $\mathsf{H}$ be a group and $\rho:\mathsf{H} \times \mathsf{H}\rightarrow \mathsf{GL}_d(k)$ a semisimple representation. There exist $C,c>0$ such that for every $w \in [\mathsf{H},\mathsf{H}]$ we have \begin{equation}\label{semisimple-1} \Big|\Big| \mu \big(\rho(1, w)\big)\Big|\Big|_{\mathbb{E}} \leq C\Big|\Big| \mu \big(\rho(w, w)\big)\Big|\Big|_{\mathbb{E}} +c.\end{equation} \end{lemma}

\begin{proof} Let $\rho^{\ast}(w,w')=\rho\big((w,w')^{-1}\big)^{t}, (w,w') \in \mathsf{H} \times \mathsf{H}$, be the dual representation of $\rho$. We may assume that the representation $\rho$ is conjugate to $\rho^{\ast}$ and $d \in \mathbb{N}$ is even, otherwise we may replace $\rho$ with the representation $\rho\times \rho^{\ast}$ and prove the statement for this representation. In particular, \begin{equation} \label{eigenvalues} \ell_i\big(\rho(w,w')\big)=\ell_i\big(\rho(w,w')^{-1}\big)=\ell_{d-i+1}\big(\rho(w,w')\big)^{-1} \end{equation} for every $(w,w') \in \mathsf{H} \times \mathsf{H}$ and $1 \leq i \leq \frac{d}{2}$. Since $\{1\} \times \mathsf{H}$ is a normal subgroup of $\mathsf{H}\times \mathsf{H}$ and $\rho$ is semisimple, Fact \ref{normal} implies that $(\wedge^s \rho)|_{\{1\} \times \mathsf{H}}$ is also a semisimple representation for every $1 \leq s \leq \frac{d}{2}$. Lemma \ref{restriction} shows that there exist finitely many group homomorphisms $\varepsilon_1,\ldots,\varepsilon_{\ell}:\mathsf{H} \rightarrow k^{\ast}$ with the following property: if $g \in \mathsf{H}$ and $\wedge^i \rho(1,g)$ is $1$-proximal and $v_{g}^{+}\in \wedge^i k^d$ is an eigenvector of $\wedge^i \rho(1,g)$ with respect to the eigenvalue of maximum modulus, there exists $1 \leq r \leq \ell$ such that \begin{equation}\label{proximal-eq*}\wedge^i \rho(h,1)v_{g}^{+}=\varepsilon_r(h)v_{g}^{+}\end{equation} for every $h \in \mathsf{H}$. By Theorem \ref{finitesubset} there exists a finite subset $F$ of $\{1\}\times \mathsf{H}$ and $C_{\rho}>0$ with the property: for every $w \in \mathsf{H}$ there exists $(1,f)\in F$ such that \begin{equation} \label{finiteeq}\max_{1 \leq i \leq d} \Big|\log \ell_i\big(\rho(1,wf)\big)-\log \sigma_i \big(\rho(1,w)\big) \Big| \leq C_{\rho}. \end{equation} Let us set $R:=C_{\rho}+\max_{g \in F}\big|\textup{det}(\rho(1,g))\big|$. Fix now $w \in [\mathsf{H},\mathsf{H}]$, observe that $\big|\textup{det}(\rho(1,w))\big|=1$, and choose $f \in \mathsf{H}$ such that $f,w \in \mathsf{H}$ satisfy (\ref{finiteeq}). If $\sum_{i=1}^{d} \big|\log \sigma_i(\rho(1,w))\big|<10Rd$ then (\ref{semisimple-1}) holds true for $c=10Rd$. Thus, we continue by assuming that $\sum_{i=1}^{d} |\log \sigma_i(\rho(1,w))|\geq 10Rd$. Then (\ref{finiteeq}) implies that $\sum_{i=1}^{d} \big|\log \ell_i ( \rho(1,wf))\big| \geq 9Rd$. By the choice of $R>0$ and the fact that $\sum_{i=1}^{d} \log \ell_i(\rho(1,wf))=\log \big| \textup{det} \rho(1,f)\big|$, we may choose $1 \leq i \leq d-1$ such that $\ell_i(\rho(1,wf))>\ell_{i+1}(\rho(1,wf))$ (otherwise we would have $\log \ell_i(\rho(1,wf))=\frac{1}{d}\log \big|\textup{det}\rho(1,f)\big|$ for every $i$, which cannot happen since $\sum_{i=1}^{d}|\log\ell_i(\rho(1,wf))|\geq 9Rd$). Therefore, $\wedge^{i}\rho(1,wf)$ is $1$-proximal and we may assume by (\ref{eigenvalues}) that $1 \leq i \leq \frac{d}{2}$. Let $v_{wf}^{+} \in \wedge^i k^d$ be a unit eigenvector of $\wedge^{i}\rho(1, wf)$ corresponding to its eigenvalue of maximum modulus. By (\ref{proximal-eq*}) we may choose $1 \leq r \leq \ell$ and write $$\wedge^i \rho\big(wf,1\big)v_{wf}^{+}=\varepsilon_r \big(wf\big)v_{wf}^{+}=\varepsilon_r(f)v_{wf}^{+},$$ since $f \in N$, $w \in [\mathsf{H},\mathsf{H}]$ and $[\mathsf{H},\mathsf{H}]\subset \textup{ker}\varepsilon_r$. We deduce that \begin{align*} \Big| \Big|\wedge^i \rho \big(wf,wf\big)\Big|\Big|& \geq \Big| \Big|\wedge^i \rho\big(wf,wf\big) v_{wf}^{+}\Big|\Big|\\ &=\Big| \Big | \wedge^i \rho(wf,1) \wedge^i \rho(1,wf) v_{wf}^{+} \Big|\Big|\\ &=\big|\varepsilon_r(f)\big|\ell_1\big(\wedge^i \rho\big(1,wf\big) \big)\\ &=\big|\varepsilon_r(f)\big| \Big({\ell_1\big(\wedge^i \rho\big(1,wf\big) \big) \ell_1\big(\wedge^i \rho\big(1,(wf)^{-1}\big) \big)}\Big)^{\frac{1}{2}}\\ &=\big|\varepsilon_r(f)\big| \Biggr({\prod_{j=1}^{i}\frac{\ell_j\big(\rho(1,wf)\big)}{\ell_{d-j+1}\big(\rho(1,wf)\big)}}\Biggr)^{\frac{1}{2}}\\ & \geq \Big(\min_{1 \leq q \leq \ell}\min_{(1,g)\in F} \big|\varepsilon_q(g)\big|\Big)  \Biggr({\frac{\ell_1\big(\rho(1,wf)\big)}{\ell_d\big(\rho(1,wf)\big)}}\Biggr)^{\frac{1}{2}}\\ &\geq e^{-C} \Big(\min_{1 \leq q \leq \ell}\min_{(1,g)\in F} \big|\varepsilon_q(g)\big|\Big)  \Biggr({\frac{\sigma_1\big(\rho(1,w)\big)}{\sigma_d\big(\rho(1,w)\big)}}\Biggr)^{\frac{1}{2}}. \end{align*} \par We conclude that there exist constants $M_1,M_2,M_3>100Rd$, depending only on the representation $\rho$, $d\in \mathbb{N}$, the finite set $F\subset \{1\} \times \mathsf{H}$ and the the constant $\min_{1 \leq q \leq \ell}\min_{(1,g)\in F} \big|\varepsilon_q(g)\big|$, such that: $$\log \frac{\sigma_1 \big(\rho(1,w)\big)}{\sigma_d \big(\rho(1,w)\big)} \leq 2\max_{1\leq i \leq \frac{d}{2}}\log \Big|\Big| \wedge^i \rho \big(wf,wf\big)\Big|\Big|+M_1\leq M_3 \log \frac{\sigma_1(\rho(w,w))}{\sigma_d(\rho(w,w))}+M_2$$ for every $w \in [\mathsf{H},\mathsf{H}]$. The proof of the lemma is complete. \end{proof}

\subsection{Proof of Theorem \ref{mainthm2}.} Let us now recall Theorem \ref{mainthm2} which states that every linear representation of a direct product of the form $\mathsf{H}\times \mathsf{H}$ is $q$-normal for some $q \in \mathbb{N}$:

\begin{theorem} \textup{(}Theorem \ref{mainthm2}\textup{)}\label{mainthm-general}  Let $k$ be a local field and $\mathsf{H}$ be a group. For every representation \hbox{$\rho:\mathsf{H} \times \mathsf{H} \rightarrow \mathsf{GL}_d(k)$} there exists $C>1$, depending only on $\rho$, with the property: for every $w_1,w_2,\ldots ,w_{r} \in \mathsf{H}\smallsetminus\{1\}$ and $r \geq d+1$ we have that $$\Big| \Big| \mu \big(\rho \big(1, [w_1,w_2,\ldots, w_{r}]\big) \big)\Big|\Big|_{\mathbb{E}} \leq 2^r C\Big(1+\sum_{i=1}^{r}\Big|\Big|\mu\big(\rho(w_i,w_i)\big)\Big|\Big|_{\mathbb{E}}\Big).$$ \end{theorem}

For the proof we need the following elementary fact.

\begin{lemma} \label{reg} Let $\mathsf{H}$ be a group and suppose that $\rho:\mathsf{H} \times \mathsf{H}\rightarrow \mathsf{GL}_d(k)$ is a representation which is $r$-normal. Then $\rho$ is $m$-normal for every $m \geq r$. In particular, there exist $\mathcal{K},a>1$, depending only on $\rho$, such that $$\Big| \Big| \rho \big(1,\big[w_1,\ldots, w_m]^{\pm1}\big)\Big|\Big| \leq \mathcal{K}^{2^{m-r}}e^{2^{m-r} a\sum_{i=1}^{r}||\mu(\rho(w_i,w_i))||_{\mathbb{E}}}$$ for every $w_1,\ldots,w_m \in \mathsf{H}$.
\end{lemma}

\begin{proof} Let us observe that there exist constants $C,a>1$, depending only on $d\in \mathbb{N}$, such that $$\Big|\Big|\rho \big(g,g\big)\Big|\Big|\leq Ce^{a||\mu(\rho(g,g)||_{\mathbb{E}}}$$ for every $g \in \Gamma$. Since $\rho$ is $r$-normal, by enlarging $C,a>1$ if necessary, we may assume that $$\Big | \Big| \rho \big(1, [w_1,\ldots,w_r]^{\pm1}\big)\Big | \Big| \leq Ce^{a \sum_{i=1}^{r}||\mu(\rho(w_i,w_i))||_{\mathbb{E}}},$$ for every $w_1,\ldots,w_r \in \mathsf{H}$. Now we may use induction to prove \begin{equation} \label{induction-ineq} \Big| \Big| \rho \big(1,\big[w_1,\ldots, w_m]^{\pm1}\big)\Big|\Big| \leq C^{-2+6\cdot 2^{m-r}}e^{2^{m-r} a\sum_{i=1}^{r}||\mu(\rho(w_i,w_i))||_{\mathbb{E}}}\end{equation} for every $w_1,\ldots,w_m \in \mathsf{H}$, $m\geq r$. Obviously (\ref{induction-ineq}) holds for $m=r$. Assume now that (\ref{induction-ineq}) holds for $m\geq r$ and let $w_1,\ldots,w_m,w_{m+1} \in \mathsf{H}$. We set $z_m:=[w_1,\ldots,w_m]$ and observe that \begin{align*} \Big | \Big| \rho(1,\big[z_m,w_{m+1}\big]\big) \Big| \Big| &=\Big | \Big| \rho \big(1,z_m^{-1}\big)\rho \big(w_{m+1}^{-1},w_{m+1}^{-1}\big)\rho \big(1,z_m\big)\rho \big(w_{m+1},w_{m+1}\big)\Big | \Big|\\ & \leq \Big | \Big|\rho \big(1,z_m^{-1}\big)\Big | \Big|\cdot \Big | \Big|\rho \big(w_{m+1}^{-1},w_{m+1}^{-1}\big)\Big | \Big| \cdot \Big | \Big|\rho \big(1,z_m\big)\Big | \Big|\cdot \Big | \Big|\rho \big(w_{m+1},w_{m+1}\big)\Big | \Big| \\ &\leq C^{-4+12\cdot 2^{m-r}} e^{2^{m-r+1}a \sum_{i=1}^{m}||\mu(\rho(w_i,w_i))||_{\mathbb{E}}}\cdot C^2 e^{2a||\mu(\rho(w_{m+1},w_{m+1}))||_{\mathbb{E}}}\\ & \leq C^{-2+6\cdot 2^{m+1-r}}e^{2^{m+1-r}a\sum_{i=1}^{m+1}||\mu(\rho(w_i,w_i))||_{\mathbb{E}}}. \end{align*} We similarly check the same bound holds for $\big | \big|\rho\big(1,[w_{m+1},z_m]\big)\big|\big|$. This completes the proof of the induction and the conclusion follows. \end{proof}

The proof Theorem \ref{mainthm-general} is based on the following key lemma.

\begin{lemma} \label{control} Let $\mathsf{H}$ be a group and $\rho:\mathsf{H} \times \mathsf{H} \rightarrow \mathsf{GL}_d(k)$ be a representation of the form: 

\begin{equation*} \label{rep} \rho \big(g,g'\big)=\begin{pmatrix}[1.3]
\rho_1(g,g') & \ldots & \ast \\ 
 0&\ddots   &\vdots  \\ 
 0&0  & \rho_{r}(g,g')
\end{pmatrix},\ \ \big(g,g'\big) \in \mathsf{H} \times \mathsf{H},\end{equation*} where \big\{$\rho_i: \mathsf{H} \times \mathsf{H} \rightarrow \mathsf{GL}_{d_i}(k) \big\}_{i=1}^{r}$ are spanning representations and $d=\sum_{i=1}^{r}d_i$. Suppose that $\rho_i$ is $2$-normal for every $1 \leq i \leq r$. Then $\rho$ is $(r+1)$-normal.\end{lemma}

Before we give the proof of Lemma \ref{control}, we will also need the following observation.

\begin{observation} \label{Obs1} Let $\psi:\mathsf{H} \times \mathsf{H} \rightarrow \mathsf{GL}_d(k)$ be a representation and $\big \{(x_1,y_1),\ldots, (x_m,y_m)\big\}$ a finite subset of $\mathsf{H} \times \mathsf{H}$. There exist $C,a>0$, depending only on $\psi$, such that for every $w \in \mathsf{H}$, $$ \max_{1 \leq i \leq m} \Big| \Big| \psi \big(wx_i,y_iw\big)^{\pm 1}\Big|\Big| \leq Ce^{a ||\mu(\psi(w,w))||_{\mathbb{E}}}.$$ \end{observation} 

\begin{proof} Observe that for every $w \in \mathsf{H}$ and $1 \leq i \leq m$ we can write $$\big(wx_i,y_i w \big)=\big(w,w\big)\big(x_i,x_i\big)\big(1,x_i\big)^{-1}\big(w,w\big)^{-1}\big(1,y_i\big)\big(w,w\big).$$ The observation now follows directly from the sub-multiplicativity of the $\ell_2$-norm $||\cdot||$.
\end{proof}

\noindent {\em Proof of Lemma \ref{control}.}
We shall use induction on the number $r \in \mathbb{N}$ of the spanning representations $\rho_1,\rho_2,\ldots,\rho_r$. First, we are going to prove the statement for $r=2$.
\medskip

\noindent {\bf Case 1:} {\em Suppose that $r=2$}. The representation $\rho$ is of the form $$\rho\big(g,g'\big)=\begin{pmatrix}[1.3]
\rho_1(g,g') & u(g,g') \\ 
 0& \rho_2(g,g') 
\end{pmatrix}, \ \ (g,g')\in \mathsf{H} \times \mathsf{H},$$ for some matrix valued function $u:\mathsf{H} \times \mathsf{H} \rightarrow \mathsf{Mat}_{d_1 \times d_2}(k)$. Observe that there \hbox{exist $R,c>0$ such that} \begin{equation} \label{trivial} \max_{i=1,2}\Big|\Big|\rho_i\big(g,g\big)\Big|\Big|\leq \Big|\Big|\rho\big(g,g\big)\Big|\Big|\leq Re^{c||\mu(\rho(g,g))||_{\mathbb{E}}}\end{equation} for every $g \in \mathsf{H}$. It is enough to prove that there exist $J,a>0$ such that $$\Big|\Big|u\big(1,[w_1,w_2]\big)\Big|\Big|\leq Je^{a\sum_{i=1}^{2} ||\mu(\rho(w_i,w_i))||_{\mathbb{E}}}$$ for every $w_1,w_2\in \mathsf{H}$. By assumption, $\rho_1, \rho_2$ are spanning representations, hence by Lemma \ref{subalgebra} there are two cases to consider:
\medskip

\noindent {\em Case 1a. $\big\langle (\rho_1 \times \rho_2)(\mathsf{H} \times \mathsf{H}) \big \rangle=\mathfrak{gl}_{d_1}(k)\times \mathfrak{gl}_{d_2}(k)$.} \newline In particular, we may choose $\big\{(a_i,b_i)\big\}_{=1}^{q} \subset \mathsf{H} \times \mathsf{H}$ and $\big\{c_i\big\}_{i=1}^{q} \subset k$ such that $$\sum_{i=1}^{q} c_i \rho_1(a_i,b_i)=\textup{I}_{d_1} \ \ \textup{and} \ \ \sum_{i=1}^{q} c_i \rho_2(a_i,b_i)=0_{d_2}.$$ \par Let $w_1,w_2 \in \mathsf{H}$ be arbitrary elements and set $w:=[w_1,w_2]$. We may directly check that \begin{align*} &\sum_{i=1}^{q}c_i \rho(wa_i,b_iw)=
\begin{pmatrix}[1.2]
\rho_1(w,w) & \mathcal{Y}(w)\\ 
 0& 0_{d_2} 
\end{pmatrix}, \\ \mathcal{Y}(w)&:=\rho_1\big(w,1\big)\Bigg(\sum_{i=1}^{q}c_i u\big(a_i,b_i\big)\Bigg)\rho_2\big(1,w\big)+ \rho_1\big(w,1\big)u\big(1,w\big)\end{align*} and note that $$\Big|\Big| \mathcal{Y}(w)\Big| \Big| \leq  \bigg|\bigg|\sum_{i=1}^{q}c_i\rho(wa_i,b_iw)\bigg|\bigg| \leq \sum_{i=1}^{q}|c_i|  \Big|\Big|\rho(wa_i,b_i w)\Big|\Big|.$$ Therefore, by Observation \ref{Obs1}, the fact that $\rho_1$ and $\rho_2$ are $2$-normal and (\ref{trivial}), there exist \hbox{$R_1,r_1>0$,} depending only on $\rho$ and $c_1,\ldots,c_q\in k$, such that  \begin{align*} \Big|\Big|u\big(1,w\big)\Big|\Big|&= \Bigg| \Bigg| \rho_1\big(w,1\big)^{-1}\mathcal{Y}(w)-\sum_{i=1}^{q} c_i u\big(a_i,b_i\big) \rho_2\big(1,w\big)\Bigg|\Bigg|\\ & \leq \Big|\Big|\rho_1\big(1,w\big)^{-1}\Big|\Big|\cdot \Bigg|\Bigg|\sum_{i=1}^{q}c_i\rho \big(wa_i,b_iw\big)\Bigg|\Bigg|+\sum_{i=1}^{q}\big|c_i\big| \Big|\Big|u\big(a_i,b_i\big)\Big|\Big| \cdot \Big|\Big|\rho_2\big(1,w\big)\Big|\Big| \\ & \leq  R_1e^{r_1\sum_{i=1}^{2} ||\mu(\rho(w_i,w_i))||_{\mathbb{E}}}.\end{align*} We deduce that there exist $R_2,r_2>0$, depending only on $\rho$, such that $$\Big|\Big|\rho \big(1,\big[w_1,w_2\big]\big)\Big|\Big|\leq R_2e^{r_2\sum_{i=1}^{2} ||\mu(\rho(w_i,w_i))||_{\mathbb{E}}}$$ for every $w_1,w_2 \in \mathsf{H}$. It follows that $\rho$ is $2$-normal and hence $3$-normal by Lemma \ref{reg}. $\qed$
\medskip

\noindent {\em Case 1b. $d_1=d_2$ and $\rho_1$ and $\rho_2$ are conjugate.} 

\noindent Up to conjugating $\rho$ by an element of $\mathsf{GL}_{d_1}(k)\times \mathsf{GL}_{d_1}(k)$ we may assume that $\rho_1=\rho_2$. By assumption, $\rho_1(\mathsf{H}\times \mathsf{H})$ spans $\mathfrak{gl}_{d_1}(k)$, hence for every $1 \leq i,j \leq d_1$ there exist $m_{ij} \in \mathbb{N}$, $\big\{c_{ij\ell}\}_{\ell=1}^{m_{ij}} \subset k$ and $\big\{(a_{ij\ell},b_{ij\ell})\big\}_{\ell=1}^{m_{ij}}\subset \mathsf{H} \times \mathsf{H}$ such that $\sum_{\ell=1}^{m_{ij}}c_{ij\ell}\rho_1(a_{ij\ell},b_{ij\ell})=E_{ij}$.
\medskip

\par Now let $w_1,w_2 \in \mathsf{H}$ and set $w:=[w_1,w_2]$. For every $1 \leq i,j \leq d_1$ we have: \begin{align*}\sum_{\ell=1}^{m_{ij}}c_{ij\ell}\rho \big(wa_{ij\ell},b_{ij\ell}w\big)&=\begin{pmatrix}[1.3]
\rho_1(w,1) & u(w,1) \\ 
0 & \rho_1(w,1)
\end{pmatrix} \begin{pmatrix}[1.3]
E_{ij} & \mathcal{U}_{ij}\\ 
0 & E_{ij}
\end{pmatrix} \begin{pmatrix}[1.3]
\rho_1(1,w) & u(1,w) \\ 
0 & \rho_1(1,w)
\end{pmatrix}\\ &= \begin{pmatrix}[1.3]
\rho_1(w,1)E_{ij}\rho_1(1,w) & \mathcal{V}_{ij}(w) \\ 
 0& \rho_1(w,1)E_{ij}\rho_1(1,w) 
\end{pmatrix}, \\
\mathcal{U}_{ij}&:=\sum_{\ell=1}^{m_{ij}}c_{ij\ell} u(a_{ij\ell},b_{ij\ell}),\\
\mathcal{V}_{ij}(w)&:=\rho_1(w,1)\mathcal{U}_{ij}\rho_1(1,w)+u(w,1)E_{ij}\rho_1(1,w)+ \rho_1(w,1)E_{ij}u(1,w). \end{align*} By Observation \ref{Obs1}, the fact that $\rho_1$ is $2$-normal and (\ref{trivial}), we obtain $B_0,b_0>0$, depending only on $\rho$, such that \begin{align*} &\Big| \Big| u(w,1)E_{ij}\rho_1(1,w)+\rho_1(w,1)E_{ij}u(1,w) \Big| \Big|= \Big| \Big| \mathcal{V}_{ij}(w)-\rho_1(w,1)\mathcal{U}_{ij}\rho_1(1,w) \Big| \Big|\\&
\leq \Big|\Big| \sum_{\ell=1}^{m_{ij}}c_{ij\ell}\rho(wa_{ij\ell},b_{ij\ell}w)\Big|\Big|+ \Big|\Big|\rho_1(1,w)\Big|\Big|\cdot \Big|\Big|\rho_1(w,1)\Big|\Big|\cdot \Big|\Big|\mathcal{U}_{ij}\Big|\Big|\\
&\leq B_0e^{b_0\sum_{i=1}^{2} ||\mu(\rho(w_i,w_i))||_{\mathbb{E}}}\end{align*} for every $1 \leq i,j \leq d_1$. Let us also observe that for every $g \in \mathsf{H}$ we have \begin{equation} \begin{split} u(g,g)&=\rho_1(g,1)u(1,g)+u(g,1)\rho_1(1,g)\\ u(g,1)&=u(g,g)\rho_1(1,g)^{-1}-\rho_1(g,1)u(1,g)\rho_1(1,g)^{-1}.\end{split} \end{equation} Therefore, since $\rho_1$ is $2$-normal and by (\ref{trivial}) we have $\big|\big|u(w,w)\big|\big|\leq Re^{c||\mu(\rho(w,w))||_{\mathbb{E}}}$, there exist $B_1,b_1>0$, depending only on $\rho$, such that $$\Big| \Big| \rho_1(w,1) u(1,w) \rho_1(1,w)^{-1}E_{ij}\rho_1(1,w)- \rho_1(w,1)E_{ij}u(1,w) \Big| \Big| \leq B_{1}e^{b_{1} \sum_{i=1}^{2} ||\mu(\rho(w_i,w_i))||_{\mathbb{E}}}.$$ In addition, by the sub-multiplicativity of the $\ell_2$-norm, note that \begin{align}\label{lowerbound0} \Big| \Big|u(1,w) \rho_1(1,w)^{-1}E_{ij}- E_{ij}u(1,w)\rho_1(1,w)^{-1} \Big| \Big| \leq \Big| \Big| \rho_1(w,1)^{-1}\Big|\Big| \cdot \Big| \Big| \rho_1(1,w)^{-1}\Big| \Big| \cdot\end{align} \begin{align*} \ \ \  \cdot \Big| \Big| \rho_1(w,1) u(1,w) \rho_1(1,w)^{-1}E_{ij}\rho_1(1,w)- \rho_1(w,1)E_{ij}u(1,w) \Big| \Big|.\end{align*} Since $\rho_1$ is $2$-normal, by using (\ref{lowerbound0}), we deduce that there exist $B_2,b_2>0$, depending only on $\rho$, such that for every $w_1,w_2 \in \mathsf{H}$ we have \begin{equation} \label{eq2}  \max_{1\leq i,j \leq d_1}\Big| \Big| \big(u(1,w) \rho_1(1,w)^{-1}\big)E_{ij}-E_{ij}\big(u(1,w)\rho_1(1,w)^{-1}\big) \Big| \Big| \leq B_{2}e^{b_{2}\sum_{i=1}^{2} ||\mu(\rho(w_i,w_i))||_{\mathbb{E}}}, \end{equation} where $w=[w_1,w_2]$. By using (\ref{eq2}) and the bound $||Q-q_{11}\textup{I}_d|| \leq d\underset{1\leq i,j\leq d}{\max}||QE_{ij}-E_{ij}Q||$ for any matrix $Q=(q_{ij})_{i,j=1}^{d}\in \mathfrak{gl}_d(k)$, we may find $B_3,b_3>0$, depending only on $\rho$, and write \begin{equation*} \label{eql1} u(1,w)\rho_1(1,w)^{-1}=\phi(w)\textup{I}_{d_1}+\mathcal{D}_{w},\end{equation*} for some $\phi(w)\in k$ and $\mathcal{D}_w \in \mathfrak{gl}_{d_1}(k)$ with $\big|\big|\mathcal{D}_w\big|\big|\leq B_3e^{b_3\sum_{i=1}^{2} ||\mu(\rho(w_i,w_i))||_{\mathbb{E}}}.$ Equivalently, since $\rho_1$ is $2$-normal, there exist $B_4,b_4>0$, depending only on $\rho$, such that \begin{equation} \label{eql1} u(1,w)=\phi(w)\rho_1(1,w)+\mathcal{D}_{w}',\end{equation} where $\mathcal{D}_w' \in \mathfrak{gl}_{d_1}(k)$ and $\big|\big|\mathcal{D}_w'\big|\big|\leq B_4e^{b_4\sum_{i=1}^{2} ||\mu(\rho(w_i,w_i))||_{\mathbb{E}}}.$ By using the fact that $$u(1,w^{-1})=-\rho_1(1,w)^{-1}u(1,w)\rho_1(1,w)^{-1},$$ $\rho_1$ is $2$-normal and (\ref{eql1}), we may choose $B_5,b_5>0$, depending only on $\rho$, and write \begin{equation} u(1,w^{-1})=-\phi(w)\rho_1(1,w)^{-1}+\mathcal{D}_{w^{-1}}'\end{equation} for some $\mathcal{D}_{w^{-1}}' \in \mathfrak{gl}_{d_1}(k)$ with $\big|\big|\mathcal{D}_{w^{-1}}'\big|\big|\leq \big| \big|\rho_1(1,w^{-1})\big| \big|^2 \big|\big|\mathcal{D}_w'\big|\big|\leq B_5e^{b_5\sum_{i=1}^{2} ||\mu(\rho(w_i,w_i))||_{\mathbb{E}}}.$ 
\medskip

\par Let us now fix $w_3 \in \mathsf{H}$. Then we have the following straightforward calculation:\begin{equation} \label{eql} \begin{split} & \rho \big(1,[w,w_3]\big)=\rho \big(1,w^{-1} \big)\rho \big(w_3^{-1},w_3^{-1}\big)\rho \big(1,w\big)\rho \big(w_3,w_3\big)\\
&=\begin{pmatrix}[1.4]
 \rho_1 \big(w_3^{-1},w^{-1}w_3^{-1}\big)&-\phi(w)\rho_1\big(w_3^{-1},w^{-1}w_3^{-1}\big)+\mathcal{D}_{w^{-1}}'\rho_1\big(w_3^{-1},w_3^{-1}\big)+\rho_1\big(1,w^{-1}\big)u\big(w_3^{-1},w_3^{-1}\big) \\ 
 0& \rho_1\big(w_3^{-1},w^{-1}w_3^{-1}\big)
\end{pmatrix}
\cdot \\ &\cdot \begin{pmatrix}[1.4]
 \rho_1 \big(w_3,w w_3\big)&\eta(w)\rho_1\big(w_3,w w_3\big)+\mathcal{D}_{w}'\rho_1\big(w_3,w_3\big)+\rho_1\big(1,w\big)u\big(w_3,w_3\big) \\ 
 0& \rho_1\big(w_3,w w_3\big)
\end{pmatrix} \\ &= 
\begin{pmatrix}[1.4]
\rho_1\big (1,[w,w_3]\big) & u\big(1,[w,w_3]\big) \\ 
0  & \rho_1\big(1,[w,w_3]\big)
\end{pmatrix},\\   & u\big(1,[w,w_3]\big):=\underline{\rho_1\big(w_3^{-1},w^{-1}w_3^{-1}\big)\phi(w)\rho_1\big(w_3,ww_3\big)} \\ &+ \Big(\mathcal{D}_{w^{-1}}'\rho_1\big(w_3^{-1},w_3^{-1}\big)+\rho_1\big(1,w^{-1}\big)u\big(w_3^{-1},w_3^{-1}\big) \Big)\rho_1\big(w_3,ww_3\big)\\ & +\rho_1\big(w_3^{-1},w^{-1}w_3^{-1}\big)\Big(\mathcal{D}_{w}'\rho_1\big(w_3,w_3\big)+\rho_1\big(1,w\big)u\big(w_3,w_3\big)\Big)\\ &-\underline{\phi(w)\rho_1 \big(w_3^{-1},w^{-1}w_3^{-1}\big)\rho_1\big(w_3,ww_3\big)}.\end{split} \end{equation} \noindent The underlined terms cancel, so $u(1,[w,w_3])$ does not depend on the scalar $\phi(w) \in k$. We recall that $\rho_1$ is $2$-normal and there exist $B_6,b_6>0$, depending only on $\rho$, such that $$\Big|\Big|\mathcal{D}_{w^{\pm 1}}'\Big|\Big|\leq B_6e^{b_6\sum_{i=1}^{2}||\mu(\rho(w_i,w_i))||_{\mathbb{E}}}.$$ We immediately check that all the remaining terms of $u(1,[w,w_3])$ are bounded by $B_7e^{b_7\sum_{i=1}^{3}||\mu(\rho(w_i,w_i))||_{\mathbb{E}}}$, where $B_7,b_7>0$ are constants depending only on $\rho$. Finally, by working similarly with the commutator $\big[w_3,[w_1,w_2]\big]$, we deduce that there exist $B_8,b_8>0$, depending only on $\rho$, such that \begin{align*} \Big|\Big| \rho \big(1,\big[[w_1,w_2],w_3\big]^{\pm 1}\big)\Big|\Big| & \leq  B_8 e^{b_8\sum_{i=1}^{3} ||\mu(\rho(w_i,w_i))||_{\mathbb{E}}}\end{align*} for every $w_1,w_2,w_3 \in \mathsf{H}$. We conclude that $\rho$ is $3$-normal. This completes the proof of the lemma when $r=2$. $ \ \square$
\medskip

\noindent {\bf Case 2:} {\em Suppose that $r \geq 3$}. The representation $\rho$ has the form $$\rho\big(g,g'\big)=
\begin{pmatrix}[1.2]
\rho_1(g, g') & v_r(g,g')  & a_{1r}(g,g')\\ 
0 & U_{r}(g,g') & V_{r}(g,g')\\ 
 0& 0 & \rho_r(g,g') 
\end{pmatrix}, \ (g,g')\in \mathsf{H} \times \mathsf{H},$$ where $v_r(g, g')$ is the collection of matrix blocks of $\rho$ different from $(1,1)$ in the first row, \hbox{$a_{1r}: \mathsf{H} \times \mathsf{H} \rightarrow \mathsf{Mat}_{d_1 \times d_r}(k)$} is the $(1,r)$-block of $\rho$, $V_{r}(g,g')$ is the collection of the blocks in the last column (except from $a_{1r}(g,g')$). The representation of $\mathsf{H}\times \mathsf{H}$ defined by $U_r$ is written in upper block form and has at most $r-2$ diagonal irreducible block representations. \par By the induction hypothesis and Lemma \ref{reg}, there exist $C_0,c_0>0$ such that \begin{equation} \label{Induction-bound}\max\big\{\big| \big| v_{r}(1,w)\big|\big|, \big|\big|U_{r}(1,w)\big| \big|,\big|\big| V_{r}(1,w)\big|\big|,\big| \big| v_{r}(w,1)\big|\big|,\big|\big|U_{r}(w,1)\big| \big|,\big|\big|V_{r}(w,1)\big|\big| \big\} \leq C_0e^{c_0 \sum_{i=1}^{r}||\mu(\rho(w_i,w_i))||_{\mathbb{E}}},\end{equation} for every $w_1,\ldots,w_{r}\in \mathsf{H}$ and $w=[w_1,\ldots, w_{r}]$. It remains to show that there exists $J,a>0$ such that \begin{equation}\label{nprove} \Big|\Big|a_{1r}\big(1,[w_1,\ldots,w_{r+1}]\big)\Big|\Big|\leq Je^{a \sum_{i=1}^{r+1}||\mu(\rho(w_i,w_i))||_{\mathbb{E}}}\end{equation} for every $w_1,\ldots, w_r,w_{r+1} \in \mathsf{H}$.\par Similarly as before, by using Lemma \ref{subalgebra}, it is enough to consider two sub-cases for the spanning representations $\rho_1$ and $\rho_r$:
\medskip

\noindent {\em Case 2a.} \hbox{$\big \langle (\rho_1\times \rho_{r})(\mathsf{H} \times \mathsf{H}) \big \rangle=\mathfrak{gl}_{d_1}(k)\times \mathfrak{gl}_{d_{r}}(k)$.}\\
 Let $w_1,\ldots,w_r\in \mathsf{H}$ be arbitrary elements and set $w:=[w_1,\ldots, w_{r}]$. We may choose $\big \{c_i\big \}_{i=1}^{q} \subset k$ and $\big \{(\gamma_i,\delta_i)\big \}_{i=1}^{q}\subset \mathsf{H} \times \mathsf{H}$ such that $$\sum_{i=1}^{q} c_{i}\rho_1\big(\gamma_i,\delta_i\big)=\textup{I}_{d_1} \ \ \textup{and} \ \ \sum_{i=1}^{q} c_{i}\rho_{r}\big(\gamma_i,\delta_i\big)=0_{d_r}.$$ Similarly as in Case 1a a direct computation shows that the $(1,r)$-block of the matrix $\mathcal{Q}(w):=\sum_{i=1}^{q} c_{i}\rho(w\gamma_i,\delta_i w)=\rho(w,1)\big(\sum_{i=1}^{q}c_i\rho(\gamma_i,\delta_i)\big) \rho(1,w)$ has the form $$\rho_1(w,1) a_{1r}(1,w)+\big(\rho_1(w,1)A+v_{r}(w,1)B\big)V_r(1,w)+ \big(\rho_1(w,1)C+v_r(w,1)D\big)\rho_r(1,w)$$ where $A,B,C,D$ are sub-blocks of the matrix $\sum_{i=1}^{q}c_i\rho(\gamma_i,\delta_i)\in \mathfrak{gl}_d(k)$. On the other hand, by using Observation \ref{Obs1}, we may find $C_1,c_1>0$, depending only on $\rho$, such that $$\Big|\Big|\mathcal{Q}(w)\Big|\Big| \leq C_1e^{c_1 ||\mu(\rho(w,w))||_{\mathbb{E}}}$$ for every $w_1,\ldots,w_{r}\in \mathsf{H}$ and $w=[w_1,\ldots, w_{r}]$. By using the fact that $\rho_1$ is $2$-normal and (\ref{Induction-bound}) we deduce that $a_{1r}$ satisfies (\ref{nprove}) and hence $\rho$ is $r$-normal. $\qed$
\medskip

\noindent {\em Case 2b. The representations $\rho_1$ and $\rho_r$ are conjugate.}\\ 
Up to conjugation by an element of $\mathsf{GL}_d(k)$ we may assume that $\rho_1=\rho_r$. For every $1 \leq i,j \leq d_1$, we choose $\big \{c_{ij\ell}\big\}_{\ell=1}^{q_{ij}} \subset k$ and $\big\{(\gamma_{ij\ell},\delta_{ij\ell}) \big\}_{\ell=1}^{q_{ij}} \subset \mathsf{H} \times \mathsf{H}$ such that $$E_{ij}=\sum_{\ell=1}^{q_{ij}} c_{ij\ell} \rho_1\big(\gamma_{ij\ell},\delta_{ij\ell}\big).$$ \par Let us fix again arbitrary elements $w_1,\ldots,w_r\in \mathsf{H}$ and $w:=[w_1,\ldots, w_{r}]$. Then if we let $\mathcal{Q}_{ij}(w):=\sum_{\ell=1}^{q_{ij}}c_{ij\ell}\rho \big(w \gamma_{ij\ell},\delta_{ij\ell}w\big)$, by Observation \ref{Obs1}, there exist $C_2,c_2>0$, depending only on $\rho$, such that $$\Big|\Big|\mathcal{Q}_{ij}(w)\Big|\Big|\leq C_2e^{c_2||\mu(\rho(w,w))||_{\mathbb{E}}}.$$ By looking the $(1,r)$-block of $\mathcal{Q}_{ij}(w)$ and using (\ref{Induction-bound}), there exist $C_3,c_3>0$, depending only on $\rho$, such that \begin{equation} \label{+2'} \Big|\Big| \rho_1(w,1)E_{ij}a_{1r}(1,w)+a_{1r}(w,1)E_{ij}\rho_1(1,w)\Big|\Big|\leq C_3 e^{c_3 \sum_{i=1}^{r}||\mu(\rho(w_i,w_i))||_{\mathbb{E}}}.\end{equation}

\par Let us observe that $$a_{1r}(w,w)=\rho_1(w,1)a_{1r}(1,w)+a_{1r}(w,1)\rho_1(1,w)+v_r(w,1)V_{r}(1,w),$$ hence, by using (\ref{Induction-bound}), there exist $C_4,c_4>0$, depending only on $\rho$, such that \begin{equation} \label{+2} \Big|\Big| \rho_1(w,1)a_{1r}(1,w)+a_{1r}(w,1)\rho_1(1,w)\Big|\Big| \leq C_4e^{c_4\sum_{i=1}^{r}||\mu(\rho(w_i,w_i))||_{\mathbb{E}}}.\end{equation}

Then, by using (\ref{+2'}) and (\ref{+2}) and the fact that $\rho_1$ is $2$-normal, we may find constants \hbox{$C_5,c_5>0$}, depending only on $\rho$, with the property \begin{equation} \max_{1\leq i,j \leq d_1}\Big|\Big| \big(a_{1r}(1,w)\rho_1(1,w)^{-1}\big)E_{ij}-E_{ij}\big(a_{1r}(1,w)\rho_1(1,w)^{-1}\big)\Big|\Big| \leq C_5e^{c_5 \sum_{i=1}^{r}||\mu(\rho(w_i,w_i))||_{\mathbb{E}}}.\end{equation} for every $w_1,\ldots, w_r\in \mathsf{H}$ and $w=\big[w_1,w_2,\ldots, w_{r}\big]$. Then we may find constants $C_6,c_6>0$ depending only on $\rho$ and write \begin{equation} \label{+1} a_{1r}\big(1,w\big)=\chi(w)\rho_1(1,w)+\Omega_{w}, \ \ \big|\big| \Omega_{w}\big| \big|\leq C_6e^{c_6\sum_{i=1}^{r}||\mu(\rho(w_i,w_i))||_{\mathbb{E}}} \end{equation} for some $\chi(w) \in k.$ In addition, let us observe that $$\rho_1(1,w)a_{1r}(1,w^{-1})+v_r(1,w)V_{r}(1,w^{-1})+a_{1r}(1,w)\rho_1(1,w^{-1})=0_{d_1},$$ hence by using the fact that $\rho_1$ is $2$-normal and (\ref{Induction-bound}), there exist $C_7,c_7>0$ depending only on $\rho$ such that \begin{equation} \label{+2''} \Big|\Big| a_{1r}(1,w^{-1})+\rho_1(1,w)^{-1}a_{1r}(w,1)\rho_1(1,w)^{-1}\Big|\Big| \leq C_7e^{c_7 \sum_{i=1}^{r}||\mu(\rho(w_i,w_i))||_{\mathbb{E}}}.\end{equation} By using (\ref{+1}) and (\ref{+2''}) we may write \begin{equation}\label{+2'''} a_{1r}\big(1,w^{-1}\big)=-\chi(w)\rho_1\big(1,w^{-1}\big)+\Omega_{w}', \ \ \big|\big| \Omega_{w}'\big| \big|\leq C_8e^{c_8\sum_{i=1}^{r}||\mu(\rho(w_i,w_i))||_{\mathbb{E}}}\end{equation} for some $C_8,c_8>0$ depending only on $\rho$. By writing $$\rho\big(1,[w,w_{r+1}]\big)=\rho \big(1,w^{-1}\big)\rho \big(w_{r+1}^{-1},w_{r+1}^{-1}\big)\rho\big(1,w \big)\rho \big(w_{r+1},w_{r+1}\big),$$ making a similar calculation as in (\ref{eql}) and using (\ref{+1}) and (\ref{+2'''}), we may check that for every $w_{r+1}\in \mathsf{H}$ the entries $a_{1r}\big(1,[w,w_{r+1}]\big)$ and $a_{1r}\big(1,[w_{r+1},w]\big)$ of $\rho(1,[w,w_{r+1}])$ do not depend on $\chi(w) \in k$ and there exist constants $C_9,c_9>0$ such that $$\Big|\Big|a_{1r}\big(1,[w_1,w_2,\ldots ,w_{r+1}]^{\pm 1}\big)\Big|\Big|\leq C_9e^{c_9 \sum_{i=1}^{r+1}||\mu(\rho(w_i,w_i))||_{\mathbb{E}}}$$ for every $w_1,w_2,\ldots,w_r \in N$. Therefore, we conclude that $a_{1r}: \mathsf{H} \times \mathsf{H} \rightarrow \mathsf{Mat}_{d_1 \times d_r}(k)$ satisfies (\ref{nprove}) for some $J,a>0$ depending only on $\rho$. In particular, Lemma \ref{reg} shows that $\rho$ is also $s$-normal for every $s \geq r+1$. The  proof of the lemma is complete. $\qed$ 
\medskip

Before we give the proof of Theorem \ref{mainthm-general} we will also need the following observation.

\begin{lemma} \label{extension} Let $k$ be a local field, $\mathsf{\Gamma}$ a group and  $\rho:\mathsf{\Gamma} \rightarrow \mathsf{GL}_{d}(k)$ a representation. There exists a finite extension $k_1$ of $k$ such that $\rho$ has the form \begin{equation} \label{generalform} \rho(\gamma)=h
\begin{pmatrix}
\rho_1(\gamma) & \cdots & \ast \\ 
 0&\ddots   &\vdots  \\ 
 0&0  & \rho_{r}(\gamma)
\end{pmatrix}h^{-1}, \ \ \gamma \in \mathsf{\Gamma}\end{equation} where \big\{$\rho_i: \mathsf{\Gamma} \rightarrow \mathsf{GL}_{d_i}(k_1) \big\}_{i=1}^{r}$ are spanning representations, $h\in \mathsf{GL}_d(k_1)$ and $d=\sum_{i=1}^{r}d_i$.\end{lemma}

\begin{proof} Let $\overline{k}$ be the algebraic closure of $k$. There exists $h\in \mathsf{GL}_d(\overline{k})$ and irreducible representations $\big\{\rho_i:\mathsf{\Gamma}\rightarrow \mathsf{GL}_{d_i}(\overline{k})\big\}_{i=1}^{r}$, where $d=\sum_{i=1}^{r}d_i$, such that $\rho$ has the form (\ref{generalform}). Let $\big\{b_{ij}:\mathsf{\Gamma} \rightarrow \overline{k}\big\}_{(i,j)\in \mathcal{F}}$ be the finite set of matrix entries off the diagonal, $a_{ij}(h^{\pm 1})\in \overline{k}$ the $(i,j)$-entry of $h^{\pm 1}\in \mathsf{GL}_d(\overline{k})$ and set $k_{0}:=k\big(\{a_{ij}(h^{\pm 1}):i,j=1,\ldots,d\}\big)$. Observe that the extension $k\subset k_0$ is finite (since $k\subset \overline{k}$ is algebraic), $\rho_{i}(\gamma)\in \mathsf{GL}_{d_i}(k_0)$ and $b_{ij}(\gamma)\in k_0$ for every $\gamma \in \mathsf{\Gamma}$. Moreover, since $\rho_1,\ldots, \rho_r$ are irreducible representations over the algebraically closed field $\overline{k}$, by Burnside's theorem (e.g. see \cite{Burnside}), the representation $\rho_i$ is spanning for every $i$. In particular, for every $1\leq i \leq r$, $1\leq p_i \leq d_i$ and $1 \leq q_i \leq d_i$, there exist $\gamma_{p_i q_i 1},\ldots ,\gamma_{p_i q_i r_i} \in \mathsf{\Gamma}$ such that $E_{p_i q_i}=\sum_{j=1}^{r_i} \epsilon_{p_i q_i j} \rho_i(\gamma_{p_i q_i j})$ and $\epsilon_{p_iq_i1},\ldots, \epsilon_{p_iq_id_i}\in \overline{k}$. Now we consider the field $$k_1:=k_0\big( \{\epsilon_{p_i q_i j}: 1\leq p_i ,q_i\leq d_i, 1\leq i \leq r\}\big)$$ and note that $k\subset k_1$ is a finite field extension. By the definition of $k_1$, $\rho_{i}:\mathsf{\Gamma} \rightarrow \mathsf{GL}_{d_i}(k_1)$ is a spanning representation for every $i$. Moreover, since $k_0\subset k_1$, $b_{ij}(\gamma)\in k_1$ for every $\gamma \in \mathsf{\Gamma}$ and every $i,j$. It follows that $k_1$ satisfies the conclusion of the lemma.\end{proof}

\noindent {\em Proof of Theorem \ref{mainthm-general}.} By Lemma \ref{extension}, we may replace $k$ with a finite extension and conjugate with an element of $\mathsf{GL}_d(k)$ so that $\rho$ has the form: $$\rho(g,g')=\begin{pmatrix}[1.3]
\rho_1(g,g') &\ldots  & \ast \\ 
 0  & \ddots   & \vdots \\ 
 0 &0  & \rho_{s}(g,g') 
\end{pmatrix}, \ \ (g,g')\in \mathsf{H}\times \mathsf{H},$$ where $d=\sum_{i=1}^{s}d_i$ and $\big\{\rho_{i}:\mathsf{H} \times \mathsf{H} \rightarrow \mathsf{GL}_{d_i}(k)\big \}_{i=1}^{s}$ are spanning representations. In particular, for every $i$, $\rho_i$ is irreducible and hence $2$-normal by Lemma \ref{mainthm-semisimple0}. Therefore, Lemma \ref{control} shows that $\rho$ is $(s+1)$-normal. In particular, by Lemma \ref{reg}, $\rho$ is $m$-normal for every $m\geq s+1$. In other words there exists $C>1$, depending only on $\rho$, such that $$\Big|\Big| \mu \big(\rho \big(1,\big[w_1,w_2,\ldots, w_{r} \big] \big)\big)\Big|\Big|_{\mathbb{E}}\leq 2^{r}C\Big(1+\sum_{i=1}^{r}\Big|\Big|\mu \big(\rho(w_i,w_i)\big)\Big|\Big|_{\mathbb{E}}\Big)$$ for every $w_1,w_2,\ldots,w_r\in \mathsf{H}$, where $r\geq d+1\geq s+1$. $\qed$
\medskip

\noindent {\em Proof of Corollary \ref{mainthm2-fiberproducts}}. Let $\rho:\Gamma \times_N \Gamma \rightarrow \mathsf{GL}_d(k)$ be a representation. Note that since $\Gamma$ is finitely generated there exists a constant $C_0>0$ such that $\big|\big| \mu(\rho(\gamma,\gamma))\big|\big|_{\mathbb{E}}\leq C_0\big|\gamma \big|_{\Gamma}$ for every $\gamma \in \Gamma$. By Theorem \ref{mainthm2} there exists $C>0$ depending only on $\rho$ such that for every $\gamma \in \Gamma$ and $w_1,\ldots,w_r \in N\smallsetminus \{1\}$, $r \geq d+1$, we have: \begin{align*}\Big|\Big| \mu \big(\rho(\gamma, \gamma[w_1,\ldots, w_r]\big)\big)\Big|\Big|_{\mathbb{E}}&\leq \Big|\Big| \mu \big(\rho(\gamma, \gamma \big)\big)\Big|\Big|_{\mathbb{E}}+ \Big|\Big| \mu \big(\rho(1,[w_1,\ldots, w_r]\big)\big)\Big|\Big|_{\mathbb{E}}\\& \leq C_0\big|\gamma \big|_{\Gamma}+2^r C+2^r C\sum_{i=1}^{r} \Big|\Big|\mu \big(\rho(w_i,w_i)\big)\Big|\Big|_{\mathbb{E}}\\ &\leq C_0\big|\gamma \big|_{\Gamma}+2^{r+1} C C_0 \sum_{i=1}^{r}\big|w_i\big|_{\Gamma}.\end{align*} The corollary follows. $\qed$ 

Similarly, by using Lemma \ref{mainthm-semisimple0}, we obtain the following corollary for semisimple representations.

\begin{corollary} \label{lemma-semisimple1} Let $\Gamma$ be a finitely generated group, $N$ a normal subgroup of $\Gamma$ and fix \hbox{$|\cdot|_{\Gamma}:\Gamma \rightarrow \mathbb{N}$} a word length function on $\Gamma$. Suppose that \hbox{$\rho:\Gamma \times _N \Gamma \rightarrow \mathsf{GL}_d(k)$} is a semisimple representation. There exist $C,c>0$, depending only on $\rho$, such that $$\Big|\Big| \mu \big(\rho(\gamma, \gamma w)\big)\Big|\Big|_{\mathbb{E}}\leq C\Big(\big |\gamma \big |_{\Gamma}+\big|\gamma w \big|_{\Gamma}\Big)+c$$ for every $\gamma \in \Gamma$ and $w \in [N,N]$. \end{corollary}

\section{Some further lemmas} \label{Lemmas}
In this section we prove some more lemmas that we need to establish Theorem \ref{nonqie1} and \hbox{Theorem \ref{nonqie2}.} We recall that $F_m$ denotes the free group on the set $\mathcal{X}:=\{x_1,\ldots,x_m\}$ and denote by $|\cdot|_{F_m}:F_m\rightarrow \mathbb{N}$ the corresponding left invariant word metric with respect to $\mathcal{X}\cup \mathcal{X}^{-1}$. The lower central series of $F_m$ is the descending series of characteristic subgroups $$\ldots \subset  \gamma_4(F_m) \subset \gamma_3(F_m) \subset \gamma_2(F_m) \subset \gamma_1(F_m):=F_m$$ inductively defined by $\gamma_{r+1}(F_m)=[F_m,\gamma_{r}(F_m)]$ for every $r\in \mathbb{N}$. The terms of the lower central series have the property that $\big[\gamma_{r}(F_m),\gamma_{d}(F_m)\big]\subset \gamma_{r+d}(F_m)$ for every $r,d\in \mathbb{N}$.

\begin{definition}\textup{(Basic commutators }\cite[Ch.3, Def. 3.3]{CMZ}\textup{)} The basic commutators of weight  are inductively defined as follows:

\noindent \textup{(1)} The elements $x_1,\ldots,x_m$ are the basic commutators of weight $1$ with the order $x_i<x_j$ if and only if $i<j$.\\ 
\noindent \textup{(2)} Suppose that commutators of weight $r\geq 1$ have been defined and ordered such that $y_1<y_2$ if the weight of $y_1$ is smaller than the weight of $y_2$. The basic commutators of weight $r+1$ are of the form $[x,y]$, where $x$ and $y$ are basic commutators of weight $k_1$ and $k_2$ respectively and the following conditions hold:\\
\noindent \textup{(i)} $k_1+k_2=r+1$ and $y<x$.\\
\noindent \textup{(ii)} if $x=[u,v]$, where $u,v$ are basic commutators with $v<u$, then $v\leq y$.\end{definition}

The weight of a basic commutator $v\in F_m$ is denoted by $w(v)\in \mathbb{N}$.
One of the key properties of the lower central series of $F_m$ is that the quotient group $A_{q}^{m}:=\gamma_q(F_m)/\gamma_{q+1}(F_m)$ is free abelian of finite rank. Magnus proved in \cite{Magnus} that the cosets defined by basic commutators of weight $q \geq 2$ in $F_m/\gamma_{q+1}(F_m)$ form a free basis for $A_{q}^{m}$ (see also \cite[Thm. 3.5]{CMZ}). We denote by $\mathcal{E}_q$ the set of basic commutators of weight $q \geq 1$ and we use the notation \hbox{$\mathcal{E}_{q}^{\pm1}:=\big\{g^{\pm 1}:g \in \mathcal{E}_q\big \}$.}

We will also use the following well known commutator identities: \begin{equation} \label{commident} \begin{split}  \big[ac,b\big]&=c^{-1}\big[a,b]c^{}\big[c,b\big]=\big[c,[a,b]^{-1}\big] \big[a,b\big] \big[c,b\big]\\  \big[a,cb\big]&=\big[a,b\big]b^{-1}\big[a,c\big]b=\big[a, b\big] \big[b, [a,c]^{-1}\big] \big[a,c\big]. \end{split} \end{equation}

By using induction on the weight of a basic commutator and the previous identities we have the following fact.

\begin{fact}\label{bci} If $v=v(x_1,\ldots,x_m)\in \gamma_r(F_m)$ is a basic commutator of weight $r \geq 1$, then $$v(x_1^n,\ldots,x_m^n)\gamma_{r+1}(F_m)=v(x_1,\ldots,x_m)^{n^{r}}\gamma_{r+1}(F_m)$$ for every $n \in \mathbb{N}$.

\end{fact}

\begin{lemma} \label{comm2} Let $F_m$ be the free group on $\{x_1,\ldots,x_m\}$, $m \geq 2$, and $w_0 \in [F_m,F_m]$. Suppose that $w \in F_m$ is an element which is a product of $\ell(w)$ elements in $\big\{x_{1}^{\pm 1},\ldots, x_{m}^{\pm 1}\big\}$.

\noindent \textup{(i)} The commutator $[w,w_0]$ is a product commutators of the form \begin{equation} \label{comm0} \Big[x_{i_1}^{\pm 1},\big[\ldots ,\big[x_{i_{r-1}}^{\pm 1},[x_{i_r}^{\pm 1},w_0\big ]^{\pm 1}\big]^{\pm 1}, \ldots \big]^{\pm 1}\Big]^{\pm 1}\in \gamma_{r+2}(F_m) \end{equation} where $1 \leq r \leq \ell(w)$ and $1 \leq i_1,\ldots, i_{r} \leq m$. In this product decomposition of $[w,w_0]$ the number of commutators of the form \textup{(}\ref{comm0}\textup{)}, with $r\in \mathbb{N}$ fixed, is equal to $\binom{\ell(w)}{r}$.
\medskip

\noindent \textup{(ii)} Fix an integer $p \geq 4$. There exists a constant $C>0$, depending only on $p\in \mathbb{N}$ and $w_0 \in [F_m,F_m]$, with the property: we may write $$[w,w_0]\gamma_{p}(F_m)=\widetilde{w}_1\cdots \widetilde{w}_M\gamma_{p}(F_m)$$ where $\widetilde{w}_1,\ldots, \widetilde{w}_M \in \bigcup_{i=3}^{p-1}\mathcal{E}_{i}^{\pm1}$ and for every $3 \leq t \leq p-1$ we have $$\textup{card}\big\{i\in [1,M]\cap \mathbb{Z}: \widetilde{w}_i\in \mathcal{E}_{t}^{\pm1} \big\} \leq C\ell(w)^{t-2}.$$\end{lemma}

\begin{proof} \textup{(i)} This part follows by using induction on $r\in \mathbb{N}$ and the commutator identities (\ref{commident}).
\medskip

\noindent \textup{(ii)} For $1 \leq s \leq p-2$, let $\mathcal{A}_{p}^{s}\subset \gamma_{s+2}(F_m)$ be the set of all commutators of the form (\ref{comm0}). By using (i) we may write $$[w,w_0]\gamma_p(F_m)=w_1\cdot \cdot \cdot w_{n}\gamma_p(F_m),$$ $w_1,\ldots,w_n \in \bigcup_{r=1}^{p-2}\mathcal{A}_{p}^{r}$ such that for every $1 \leq r \leq p-2$ the cardinality of the set $\big \{i\in [0,n]\cap \mathbb{Z}: w_{i} \in \mathcal{A}_{p}^{r}\big \}$ is at most equal to $\binom{\ell(w)}{r}$. Note that $\mathcal{A}_{p}^{s}$ is a finite subset of $F_m$, so there exists $C_0=C_0(p,w_0)>1$ with the property: if $w_r \in \mathcal{A}_p^{s}$, for some $r$, we may write $w_r=\widetilde{w}_{r1}\cdots \widetilde{w}_{rd_r}\omega_{r}$ for some $\widetilde{w}_{r1},\ldots, \widetilde{w}_{r d_r}\in \bigcup_{j=s+2}^{p-1}\mathcal{E}_{j}^{\pm 1}$, $\omega_r \in \gamma_p(F_m)$ and $d_r \leq C_0$. Therefore, the total number of elements in $\bigcup_{i=3}^{t}\mathcal{E}_{i}^{\pm1}$ from the decompositions of $\big \{\widetilde{w}_{r1}\cdots \widetilde{w}_{rd_r}\omega_{r}:  1 \leq r \leq n\big \}$ is \hbox{at most equal to} $$\sum_{i=1}^{t-2} \sum_{\{r: w_r \in \mathcal{A}_{p}^{i}\}}C_0 \leq \sum_{i=1}^{t-2} C_0\binom{\ell(w)}{i} < tC_0 \ell(w)^{t-2}.$$ Now the conclusion follows by observing that $\prod_{r=1}^{n}\big(\widetilde{w}_{r1}\cdots \widetilde{w}_{rd_r}\big)\gamma_{p}(F_m)=[w,w_0]\gamma_p(F_m)$.\end{proof}

We also need the following lemma which is the content of \cite[Lem. 1]{Hidber}.

\begin{lemma} \label{rewrite} \textup{(Hidber \cite[Lem. 1]{Hidber})} Let $p \geq 3$ and $\mathcal{E}_i=\{e_{i1},\ldots, e_{in_i}\}$ be the set of basic commutators of weight $1 \leq i \leq p-1$. Suppose that $w \in \gamma_{i}(F_m)$ and $w\gamma_p(F_m)=w_1\cdot \cdot \cdot w_n \gamma_p(F_m)$ where $w_{1},\ldots,w_n \in \bigcup_{j=i}^{p-1} \mathcal{E}_{i}^{\pm 1}$. For every $i \leq j \leq p-1$ set \hbox{$L_j:=\textup{card}\big \{r\in [0,n]\cap \mathbb{Z}: w_{r}\in \mathcal{E}_j^{\pm 1}\big \}.$} Then we can write $$w\gamma_p(F_m)=e_{i1}^{s_1}\cdot \cdot \cdot e_{in_i}^{s_{n_i}} w_{1}'w_{2}'\cdot \cdot \cdot w_{R}' \gamma_p(F_m)$$ for some $s_1,\ldots,s_{n_i} \in \mathbb{Z}$ and $w_{1}',\ldots,w_{R}'\in \bigcup_{j=i+1}^{p-1}\mathcal{E}_{j}^{\pm 1}$. Moreover, there is $D>1$, depending only on $m,p \in \mathbb{N}$, such that $$\textup{card}\big \{r \in [1,R]\cap \mathbb{Z}: w_{r}'\in \mathcal{E}_{j}^{\pm 1} \big \}\leq \sum_{q=0}^{\left \lfloor \frac{j-1}{i}  \right \rfloor} D^q L_{i}^{q} L_{j-iq}.$$ 
\end{lemma}

Now let us fix a word length function $|\cdot|_{A_{p}^m}:A_{p}^m\rightarrow \mathbb{N}$ on the free abelian group $A_{p}^m=\gamma_p(F_m)/\gamma_{p+1}(F_m)$. We shall use the previous two lemmas to prove the following lemma.

\begin{lemma} \label{estimate1} Let $p \geq 4$ and $m \geq 2$ be integers and $\mathcal{B}$ be a finite subset of $[F_m,F_m]$. There exists $c>0$, depending only on $p,m\in \mathbb{N}$ and $\mathcal{B}$, with the following property: if $w \in \gamma_{p}(F_m)$ is written as a product of the form $w=\prod_{i=1}^{M} z_{i}^{-1}w_iz_{i}$ where $w_1,\ldots,w_M \in \mathcal{B}$, then $$\max \big \{ M,|z_1|_{F_m}, \ldots, |z_{M}|_{F_m} \big\} \geq c \big| w \gamma_{p+1}(F_m)\big|_{A_{p}^m}^{\frac{1}{p-1}}.$$\end{lemma}

\begin{proof} Let us set $R_{M}:=\max \big \{ M,|z_1|_{F_m}, \ldots, |z_{M}|_{F_m} \big\}$. We may write \begin{equation} \begin{split} w \gamma_{p+1}\big(F_m\big)=\big[z_{1},w_1^{-1}\big]w_1 \big[z_2,w_2^{-1}\big] w_2 \cdot \cdot \cdot [z_M,w_{M}^{-1}\big]w_M \gamma_{p+1}(F_m). \end{split}\end{equation} By Lemma \ref{comm2} (ii) there exists a constant $C>0$, depending only on the finite set $\mathcal{B}$ and $p\in \mathbb{N}$, with the following properties:

\medskip

\noindent \textup{(i)} for each $1 \leq j \leq M$, we may write $[z_{j},w_{j}^{-1}]\gamma_{p+1}(F_m)=\widetilde{z_j}\gamma_{p+1}(F_m)$ such that $\widetilde{z}_{j}$ is a product of elements in $\bigcup_{r=3}^{p-1}\mathcal{E}_{r}^{\pm1}$. In this product decomposition of $\widetilde{z}_j$, for $3 \leq t \leq p-1$, the number of elements in $\mathcal{E}_{t}^{\pm1}$ is at most equal to $C\big|z_i\big|_{F_m}^{t-2}$.

\noindent \textup{(ii)} for each $1 \leq j \leq M$, we may write $w_j \gamma_{p+1}(F_m)=\widetilde{w}_j\gamma_{p+1}(F_m)$ so that $\widetilde{w}_j$ is a product of at most $C>0$ elements in $\bigcup_{r=2}^{p-1}\mathcal{E}_{r}^{\pm1}$.
\medskip

In particular, by using (i) and (ii), we may write $w\gamma_{p+1}(F_m)=v_{12}\cdot\cdot \cdot v_{d_2 2}\gamma_{p+1}(F_m),$ where $v_{12},\ldots, v_{d_22} \in \bigcup_{r=2}^{p-1}\mathcal{E}_r^{\pm1}$ and the total number of elemets in $\mathcal{E}_{t}^{\pm1}$ \hbox{satisfies the upper bound} \begin{align} \label{numberofcomm} L_{t,2}&:=\textup{card} \big \{i \in [0,d_2]\cap \mathbb{Z}:v_{i2} \in \mathcal{E}_{t}^{\pm 1}\big \}\leq  CM+\sum_{i=1}^{s} C\big|z_i\big|_{F_m}^{t-2}\\ \nonumber &\leq CM+\sum_{i=1}^{M} CR_{M}^{t-2}\leq 2C R_M^{t-1}. \end{align}

Now we repeatedly apply Lemma \ref{rewrite} to the coset $w\gamma_{p+1}(F_m)$. By applying Lemma \ref{rewrite} to $w\gamma_{p+1}(F_m)=v_{12}\cdot\cdot \cdot v_{d_2 2}\gamma_{p+1}(F_m)$ and using the fact that $w \in \gamma_{p}(F_m)$, we may write $$w\gamma_{p+1}(F_m)=v_{13}\cdot \cdot \cdot v_{d_3 3}\gamma_{p+1}(F_m)$$ where $v_{13},\ldots,v_{d_3 3}\in \bigcup_{j=3}^{p}\mathcal{E}_{j}^{\pm 1}$ and for every $3 \leq t \leq p$, $L_{t,3}:=\textup{card} \big \{ i \in [0,d_3]\cap \mathbb{Z}:v_{i3}\in \mathcal{E}_{t}^{\pm1}\big \}$ satisfies the upper bound \begin{align*} L_{t,3} \leq \sum_{q=0}^{\left \lfloor \frac{t-1}{2}  \right \rfloor}D^q L_{2,2}^{q}L_{t-2q,2}&\leq \sum_{q=0}^{\left \lfloor \frac{t-1}{2} \right \rfloor} D^q \big(2CR_{M}\big)^{q} 2CR_{M}^{t-2q-1}\\ & \leq 2C R_M^{t-1} \sum_{q=0}^{\left \lfloor \frac{t-1}{2}  \right \rfloor}(2CDR_M^{-1})^q\leq \big(2DC\big)^{\frac{p}{2}}R_{M}^{t-1}.\end{align*} By continuing inductively and using Lemma \ref{rewrite}, we deduce that for every $3 \leq r \leq p$ \hbox{we can write} $$w\gamma_{p+1}(F_m)=v_{1r}\cdot \cdot \cdot v_{d_r r} \gamma_{p+1}(F_m),$$ where $v_{1r},\ldots,v_{d_r r}\in \bigcup_{j=r}^{p}\mathcal{E}_{j}^{\pm 1}$, and for $r \leq t \leq p$, $L_{t,r}:=\textup{card} \big \{i \in [0,d_r]\cap \mathbb{Z}: v_{i r}\in \mathcal{E}_{i}^{\pm 1}\big \}$ satisfies the upper bound: \begin{align*} L_{t,r}& \leq \sum_{q=0}^{\left \lfloor \frac{t-1}{r-1} \right \rfloor} D^q L_{r-1,r-1}^{q} L_{t-q(r-1),r-1}\leq \sum_{q=0}^{\left \lfloor \frac{t-1}{r-1}  \right \rfloor} D^q \big(C_{r-1}R_M^{r-2}\big)^{q}\big(C_{r-1}R_M^{t-(r-1)q-1}\big)\\ &\leq C_{r-1}R_M^{t-1}\sum_{q=0}^{\left \lfloor \frac{t-1}{r-1}  \right \rfloor} \big(DC_{r-1}R_M^{-1}\big)^{q}\leq R_M^{t-1}C_{r},\end{align*} where $C_{r}:=(DC_{r-1})^{\frac{p}{2}}$ and $C_{2}=2C$. It follows that we can write $$w \gamma_{p+1}(F_m)=v_{1p}\cdot\cdot \cdot v_{d_{p} p}\gamma_{p+1}(F_m)$$ where $v_{1p},\ldots,v_{d_p p}\in \mathcal{E}_{p}^{\pm 1}$ and $$d_{p}\leq C_p R_{M}^{p-1}\leq (2C)^{(\frac{p}{2})^{p-2}}D^{(1+\frac{p}{2})^{p-1}}R_{M}^{p-1}.$$ In particular, set $c_1:=\max \big\{ \big|g\gamma_{p+1}(F_m)\big|_{A_{p}^m}: g \in \mathcal{E}_{p}^{\pm1} \big\}$, note that $\big| w\gamma_{p+1}(F_m)\big|_{A_{p}^m}\leq d_p c_1$, hence \begin{equation}R_M \geq c\big| w\gamma_{p+1}(F_m)\big|_{A_{p}^m}^{\frac{1}{p-1}} \end{equation} where $c:=c_1^{-\frac{1}{p-1}}(2C)^{-\frac{1}{p-1}(\frac{p}{2})^{p-2}}D^{-\frac{1}{p-1}(1+\frac{p}{2})^{p-1}}$. The proof of the lemma is complete. \end{proof}
\medskip

By using Lemma \ref{estimate1} we obtain the following lemma that we need for the proof of \hbox{Theorem \ref{nonqie1}.}

\begin{lemma}\label{estimate2} Let \hbox{$v=\big[b_1,b_2\big]$} be a basic commutator of weight $r \geq 2$, where $b_1,b_2\in F_m$ are basic commutators with $b_2<b_1$ and $w(b_1)+w(b_2)=r$. Suppose that $\mathcal{A}$ is a finite subset of $[F_m,F_m]$ such that the normal subgroup $\langle \langle \mathcal{A}\rangle \rangle$ of $F_m$ contains the element $v(x_1^n,\ldots,x_m^n)$ for every $n\in \mathbb{N}$. We consider the fiber product of $F_m$ \hbox{with respect to $\langle \langle \mathcal{A} \rangle \rangle$} $$\Delta_{\mathcal{A}}:=F_m \times_{\langle \langle \mathcal{A}\rangle \rangle} F_m=\Big \langle \big \{ (x_i,x_i), (1,w):1\leq i \leq m, w \in \mathcal{A}\big\}\Big \rangle$$ and fix a word length function $|\cdot|_{\mathcal{A}}:\Delta_{\mathcal{A}} \rightarrow \mathbb{N}$. Define the element of $\langle \langle \mathcal{A}\rangle \rangle$\textup{:} $$V_q\big(x_1^n,\ldots, x_m^n \big):=\Big[\big[v(x_1^n,\ldots,x_m^n),b_2(x_1^n,\ldots,x_m^n)\big],\underbrace{v(x_1^n,\ldots,x_m^n), \ldots, v(x_1^n,\ldots,x_m^n)}_{q-\textup{times}}\Big], \ \ n  \in \mathbb{N}.$$ There exists $C>0$ such that for every $n \in \mathbb{N}$ we have $$\Big| \big(1, V_q\big(x_1^n,\ldots, x_m^n\big)\big)\Big|_{\mathcal{A}}\geq Cn^{1+\varepsilon}$$ where $\varepsilon=\frac{1}{rq+r+w(b_2)-1}$.\end{lemma}

\begin{proof} Without loss of generality we may assume that the word length function $|\cdot|_{\mathcal{A}}:\Delta_{\mathcal{A}}\rightarrow \mathbb{N}$ is induced by the finite generating subset $\{(x_i,x_i)^{\pm 1},(1,w)^{\pm 1}: 1 \leq i \leq m, w \in \mathcal{A}\}$ of $\Delta_{\mathcal{A}}$. Let us set $R_n:=\big| \big(1, V_q(x_1^n,\ldots, x_m^n)\big)\big|_{\mathcal{A}}$. There exist $a_1,\ldots,a_{s+1}\in F_m$ and $w_1,\ldots, w_{s}\in \mathcal{A}^{\pm1}$: \begin{equation} \label{fiberpr} \big(1,V_q\big(x_1^n,\ldots, x_m^n\big)\big)=\big(a_1,a_1 \big)\big(1,w_{1}\big)\cdot \cdot \cdot \big(a_s,a_s\big) \big(1,w_s\big) \big(a_{s+1},a_{s+1}\big)\end{equation} and also:

\noindent \textup{(i)} for $1 \leq i \leq s$, $a_i \in F_m$ is written as product of at most $\ell(a_i)$ elements in $\big\{x_1^{\pm1},\ldots,x_m^{\pm1}\big\}$,\\
\noindent \textup{(ii)} $a_1\cdot \cdot \cdot a_{s+1}=1$ and $w_1,\ldots, w_s \in \mathcal{A}^{\pm 1}$,\\
\noindent \textup{(iii)} $R_n=\sum_{i=1}^{s+1}\ell(a_i)+s$. 
\medskip

In particular, we may write \begin{equation} \begin{split} V_q\big(x_1^n,\ldots, x_m^n\big)=\prod_{i=1}^{s} \big(a_1\cdots a_i) w_i (a_1 \cdots a_i)^{-1}, \end{split}\end{equation} and we observe from (iii) that $s+\big|a_1\cdots a_i \big|_{F_m} \leq R_n$. Let us set $\beta:=rq+r+w(b_2)$ and observe that $V_q(x_1^n,\ldots,x_m^n)\in \gamma_{\beta}(F_m)$. By applying Lemma \ref{estimate1} for $p=\beta$, there exists a constant $C_1>0$, depending only on the finite set $\mathcal{A}$ and $q\in \mathbb{N}$, such that for every $n \in \mathbb{N}$ \begin{equation} \label{upperbound'} R_{n} \geq C_1 \Big| V_q\big(x_1^n,\ldots, x_m^n \big)\gamma_{\beta+1}(F_m)\Big|_{A_{\beta}^m}^{\frac{1}{\beta-1}}. \end{equation}

\par Now  we shall use a technique from \cite{BMS} to bound the right hand side of (\ref{upperbound'}). First, by using the commutator identities (\ref{commident}) and Fact \ref{bci} for the basic commutators $v,b_2\in F_m$, we may directly check that for every $n\in \mathbb{N}$ we have $$V_q\big(x_1^n, \ldots, x_m^n \big)\gamma_{\beta+1}(F_m)=V_q\big(x_1,\ldots ,x_m\big)^{n^{\beta}}\gamma_{\beta+1}(F_m).$$ Note also that the element $V_q(x_1,\ldots,x_m)\in \gamma_{\beta}(F_m)$ is a basic commutator of weight exactly $\beta$ since $[v,b_2]$ is a basic commutator, $b_2<v$ and $v<[v,b_2]$. In particular, the coset $V_q\big(x_1,\ldots,x_m\big)\gamma_{\beta+1}(F_m)$ is an infinite order element of the free abelian group $A_{\beta}^{m}=\gamma_{\beta}(F_m)/\gamma_{\beta+1}(F_m)$. Therefore, there exists $C_2>0$ such that \begin{equation} \label{lowerbound} \Big| V_q\big(x_1^n,\ldots, x_m^n \big)\gamma_{\beta+1}(F_m)\Big|_{A_{\beta}^m} \geq C_2 n^{rq+r+w(b_2)} \end{equation} for every $n \in \mathbb{N}$. By combining (\ref{upperbound'}) and (\ref{lowerbound}) we conclude that there exists $C_3>0$ such that $$R_n \geq C_3n^{1+\frac{1}{rq+r+w(b_2)-1}}$$ for every $n \in \mathbb{N}$.  The lemma follows.
\end{proof}

We will also need the following lemma for the proof of Theorem \ref{nonqie2} in \S\ref{proofs2}.

\begin{lemma} \label{estimate3}  Let $F_{m+p}$ \textup{(}resp. $F_m$\textup{)} be the free group on $\mathcal{S}_{m+p}:=\{x_1,\ldots,x_m\} \cup \{y_1,\ldots,y_p\}$  \textup{(}resp. $\mathcal{S}_m:=\{x_1,\ldots,x_m\}$\textup{)}. We consider the following data:\\
\noindent \textup{(a)} \hbox{$v(x_1,\ldots,x_m)=\big[b_1,b_2\big]$} is a basic commutator of weight $r \geq 2$, where $b_1,b_2\in F_m$ are basic commutators and $b_2<b_1$.\\
\noindent \textup{(b)} $\mathcal{B}=\big\{\omega_1,\ldots, \omega_{s}\big\}$ is a finite subset of $[F_m,F_m]$ such that the normal subgroup $\langle \langle \mathcal{B}\rangle \rangle$ of $F_m$ contains the element $v(x_1^n,\ldots,x_m^n)$ for every $n\in \mathbb{N}$.\\
\noindent \textup{(c)} $\Gamma_{\mathcal{B}}$ is a Gromov hyperbolic group with a presentation of the form \begin{equation} \label{estimate3-pr} \Gamma_{\mathcal{B}}=\Big \langle x_1,\ldots,x_m, y_1,\ldots, y_{p}\ \big| \ \omega_{1}R_1,\ldots, \omega_{s}R_s, R_{s+1}, \ldots, R_{s+\ell} \Big \rangle \end{equation} where $R_1,\ldots,R_{s+\ell} \in \langle \langle y_1,\ldots,y_p \rangle \rangle$. \par Denote by $\overline{x}_i,\overline{y}_j$ be the image of $x_i,y_j$ in $\Gamma_{\mathcal{B}}$ respectively. Let $P_{\mathcal{B}}$ be the fiber product of $\Gamma_{\mathcal{B}}$ with respect to the normal subgroup $N_{\mathcal{B}}=\langle \langle \overline{y}_1,\ldots,\overline{y}_p \rangle \rangle$, $$P_{\mathcal{B}}=\Big \langle \big(\overline{x}_1,\overline{x}_1\big),\ldots, \big(\overline{x}_m,\overline{x}_m\big), \big(1,\overline{y}_1\big),\big(\overline{y}_1,1\big),\ldots, \big(1,\overline{y}_p\big),\big(\overline{y}_r,1\big)\Big \rangle$$ and fix a word length function $|\cdot|_{\mathcal{B}}:P_{\mathcal{B}} \rightarrow \mathbb{N}$. Finally, consider the element of $N_{\mathcal{B}}$: $$V_q\big(\overline{x}_1^n,\ldots, \overline{x}_m^n \big):=\Big[\big[v(\overline{x}_1^n,\ldots,\overline{x}_m^n),b_2(\overline{x}_1^n,\ldots, \overline{x}_m^n)\big],\underbrace{v(\overline{x}_1^n,\ldots,\overline{x}_m^n), \ldots, v(\overline{x}_1^n,\ldots,\overline{x}_m^n)}_{q-\textup{times}}\Big], \ n \in \mathbb{N}.$$ There exists $C>0$ such that for every $n \in \mathbb{N}$, $$\Big| \big(1, V_q\big(\overline{x}_1^n,\ldots, \overline{x}_m^n \big)\big)\Big|_{\mathcal{B}}\geq Cn^{1+\varepsilon}$$ where $\varepsilon=\frac{1}{rq+r+w(b_2)-1}$.\end{lemma}

\begin{proof} Let $\mathcal{R}$ be the set of relations in the given presentation of $\Gamma_{\mathcal{B}}=\langle \mathcal{S}_{m+p}|\mathcal{R} \rangle$ and denote by \hbox{$\pi:F_{m+p}\twoheadrightarrow{} \Gamma_{\mathcal{B}}$} the canonical projection. Let also $\eta:F_{m+p}\twoheadrightarrow{} F_m$ be the retract onto $F_m=F(\mathcal{S}_m)$ (i.e. $\eta|_{F_m}$ is the identity homomorphism) and note that $\eta(\mathcal{R})\subset  \mathcal{B}^{\pm 1}\cup\{1\}$. Let us also observe that since \hbox{$v(x_1^n,\ldots,x_m^n)\in \langle \langle \mathcal{B} \rangle \rangle$} and $\pi(\langle \langle \mathcal{B} \rangle \rangle) \subset N_{\mathcal{B}}$, we verify that $V_q(\overline{x}_1^m,\ldots, \overline{x}_m^n)$ is an element of $N$. We may also assume that the word metric on $P_{\mathcal{B}}$ is induced by the generating set $\big\{(\overline{x}_i,\overline{x}_i)^{\pm 1}, (1,\overline{y}_j)^{\pm 1}: x_i,y_j\in \mathcal{S}_{m+p}\big\}$ and let us set $$E_{n}:=\Big|\big(1, V_q\big(\overline{x}_1^n,\ldots, \overline{x}_m^n \big)\big)\Big|_{\mathcal{B}}.$$ \par By the definition of the metric on $P_{\mathcal{B}}$, we may find $\overline{w}_1,\ldots,\overline{w}_T \in \big\{\overline{y}_1^{\pm 1},\ldots, \overline{y}_p^{\pm 1} \big\}$, \hbox{$\overline{z}_1,\ldots,\overline{z}_{T+1} \in \Gamma_{\mathcal{B}}$} and $z_1,\ldots,z_{T+1} \in F_{m+p}$, with $\pi(z_i)=\overline{z}_i$ for every $i$, such that:
\medskip

\noindent \textup{(i)} $\big(1, V_q\big(\overline{x}_1^n,\ldots, \overline{x}_m^n \big)\big)=\big(\overline{z}_1,\overline{z}_1\big)(1, \overline{w}_1)\cdots \big(\overline{z}_T,\overline{z}_T\big)(1,\overline{w}_T)\big(\overline{z}_{T+1},\overline{z}_{T+1}\big)$,\\
\noindent \textup{(ii)} $\overline{z}_1\cdots \overline{z}_{{T+1}}=1$,\\
\noindent \textup{(iii)} for each $1 \leq i\leq T+1$, $z_i$ is a product of at most $\ell(z_i)$ elements in $\mathcal{S}_{m+p}^{\pm1}$ and also $E_n=\sum_{i=1}^{T+1} \ell(z_i)+T$.
\medskip

Now let us choose lifts $w_1,\ldots,w_{T}\in \big\{y_1^{\pm 1},\ldots, y_{p}^{\pm}\big\}$ of $\overline{w}_1,\ldots, \overline{w}_T\in N_{\mathcal{B}}$ respectively. Note that the following two elements of $F_{m+p}$, $$W_{n}:=\big(z_1 w_1\cdots z_{T}w_{T}z_{T+1}\big)^{-1}V_q\big(x_1^n,\ldots, x_m^n \big) \ \ \textup{and} \ \ W_{n}':=z_1\cdots z_{T+1},$$ represent the identity in $\Gamma$ and observe that for every $n\in \mathbb{N}$ we have \begin{align*} \max \Big\{\big|W_n\big|_{F_{m+p}}, \big|W_n'\big|_{F_{m+p}}\Big\}\leq E_{n}+10^{q}n\big(\big|v\big|_{F_m}+\big|b_2\big|_{F_m}\big)=:E_n'. \end{align*} Moreover, since $\Gamma_{\mathcal{B}}= \langle \mathcal{S}_{m+p} \big| \mathcal{R}  \rangle$ is Gromov hyperbolic, by \cite[Thm. 2]{OS}, the fiber product $F_{m+p} \times_{\langle \langle \mathcal{R} \rangle \rangle} F_{m+p}$ is an undistorted subgroup of $F_{m+p} \times F_{m+p}$. In particular, we may choose $C_1>1$ and sequences $\big\{g_{i}\big \}_{i=1}^{s_n+1}, \big\{h_{i}\big \}_{i=1}^{d_n+1} \subset F_{m+p}$ and $\big\{U_{i}\big \}_{i=1}^{s_n}, \big\{Y_{i}\big \}_{i=1}^{d_n} \subset \mathcal{R}^{\pm 1}$ such that \begin{align*}\big(1,W_n \big)&=\big(g_1,g_1 \big)\big(1,U_1\big)\cdots \big(g_{s_n},g_{s_n}\big)\big(1,U_{s_n}\big)\big(g_{s_n+1},g_{s_n+1}\big)\\
\big(1,W_n'\big)&=\big(h_1,h_1\big)\big(1,Y_1\big)\cdots \big(h_{d_n},h_{d_n}\big)\big(1,Y_{d_n}\big)\big(h_{d_n+1},h_{d_n+1}\big)\end{align*} where $\sum_{i=1}^{s_n+1}\big|g_i\big|_{F_{m+p}}+s_n\leq C_1E_n'$ and $\sum_{i=1}^{d_n+1}\big|h_i\big|_{F_{m+p}}+s_n'\leq C_1E_n'$ for every $n \in \mathbb{N}$. Therefore, we may write: \begin{align} \label{estimate3-eq1}V_q\big(x_1^n,\ldots, x_m^n \big)&=z_1 w_1\cdots z_{T}w_{T}z_{T+1} \prod_{i=1}^{s_{n}}u_{i}U_{i}u_{i}^{-1}, u_{i}:=g_1\cdots g_{i} \\ z_1\cdots z_{T+1}&=\prod_{i=1}^{d_{n}} v_{i} Y_{i} v_{i}^{-1}, \  v_i=h_1\cdots h_i,\end{align} where $V_q\big(x_1^n,\ldots, x_m^n \big) \in F_{m}$ is defined in Lemma \ref{estimate2}. Moreover, observe that $$\big|u_i\big|_{F_{m+p}} \leq C_1 E_n', \ \big|v_j\big|_{F_{m+p}}\leq C_1 E_n', \ \textup{and} \ \max\big \{s_n,d_n\big\} \leq C_1E_n'$$ for every $1\leq i\leq s_n$ and $1\leq j\leq d_n$ and $n \in \mathbb{N}$. By using (\ref{estimate3-eq1}) and applying the retract $\eta:F_{m+p}\twoheadrightarrow F_m$ and since $w_1,\ldots, w_{T+1}\in \textup{ker}\eta$, we may write $$V_q\big(x_1^n,\ldots, x_m^n \big)=\prod_{i=1}^{d_{n}} \eta(v_{i}) \eta(Y_{i}) \eta(v_{i})^{-1} \prod_{i=1}^{s_{n}}\eta(u_{i}) \eta(U_{i})\eta(u_{i})^{-1},$$ where $\eta \big(U_i\big),\eta(Y_j) \in \mathcal{B}^{\pm 1} \cup\{1\}$. Note that the word length of $\eta(u_i)$ and $\eta(v_j)$ in $F_m$ is at most equal to $C_1E_n'$. By applying Lemma \ref{estimate1} for $V_q(x_1^n,\ldots,x_m^n)\in \gamma_{\beta}(F_m)$, where $\beta=rq+r+w(b_2)$, we conclude that there exists a constant $C_2>1$, independent of $n$, such that \begin{align*}C_1 E_n' &\geq \max_{i,j}\Big\{s_n,d_n,\big|u_i\big|_{F_{m+p}},\big|v_j\big|_{F_{m+p}}\Big\}\\  &\geq \max_{i,j}\Big\{s_n,d_n,\big|\eta(u_i)\big|_{F_m},\big|\eta(v_j)\big|_{F_{m}}\Big\} \geq C_2 \Big|V_q(x_1^n,\ldots,x_m^n\big)\gamma_{\beta+1}(F_m)\Big|^{\frac{1}{\beta-1}}_{A_{\beta}^{m}}\end{align*} for every $n \in \mathbb{N}$. Now by using (\ref{lowerbound}) we deduce that there exists $C_3>0$ such that \hbox{$E_n' \geq C_3 n^{1+\frac{1}{\beta-1}}$} for every $n \in \mathbb{N}$. In particular, there exists $C_4>0$ such that $$E_{n} \geq C_4n^{1+\frac{1}{rq+r+w(b_2)-1}}$$for every $n\in \mathbb{N}$. The proof of the lemma is complete.\end{proof}

\section{The Rips construction} \label{Rips-construction}
In this section we recall Rips' construction from \cite{Rips}. For other generalizations of the Rips construction we refer the reader to \cite{Rips-Wise}, \cite[Thm. 1.6]{HW}, \cite{BO} and \cite{Arenas}. For the definition of the $C'(\lambda)$, $\lambda>0$, small cancellation condition see \cite[Ch. V]{LS}. Rips established the following theorem in order to provide pathological examples of finitely generated subgroups of small cancellation groups. 

\begin{theorem}\textup{(Rips} \cite{Rips}\textup{)} \label{Rips}Let $Q$ be a finitely presented group with presentation $$Q=\big \langle a_1,\ldots,a_m \ \big|\  R_1, \ldots, R_{\ell} \big \rangle$$ where $R_1,\ldots,R_{\ell}$ are \textup{(}possibly empty\textup{)} words of the free group on $\{a_1,\ldots,a_m\}$. There exist explicit choices of sequences of integers $\{r_i\}_{i=1}^{\ell}, \{s_i\}_{i=1}^{\ell}$ and $\{\kappa_{qi}\}_{i=1}^{m}, \{\lambda_{qi}\}_{i=1}^{m}$, $q=1,\ldots,4$, such that the finitely presented group  \begin{equation} \label{Rips group}\Gamma=\sbox0{$a_1,\ldots, a_m, x,y  \ \Biggr | \  \begin{matrix}[1.1]
 R_1xy^{r_1}xy^{r_1+1}\cdots xy^{r_1+s_1}, \ldots, R_{\ell}xy^{r_{\ell}}xy^{r_{\ell}+1}\cdots xy^{r_{\ell}+s_{\ell}}, \ \ \ \ \ \ \ \ \ \ \ \ \ \ \ \\
a_{i}xa_{i}^{-1}xy^{\lambda_{1i}}xy^{\lambda_{1i}+1}\cdots xy^{\kappa_{1i}+\lambda_{1i}},\ a_{i}^{-1}xa_{i}xy^{\lambda_{2i}}xy^{\lambda_{2i}+1}\cdots xy^{\kappa_{2i}+\lambda_{2i}},\\
a_{i}ya_{i}^{-1}xy^{\lambda_{3i}}xy^{\lambda_{3i}+1}\cdots xy^{\kappa_{3i}+\lambda_{3i}},\ a_{i}^{-1}ya_{i}xy^{\lambda_{4i}}xy^{\lambda_{4i}+1}\cdots xy^{\kappa_{4i}+\lambda_{4i}}\\ 
i=1,\ldots,m \end{matrix}$}
\mathopen{\resizebox{1.5\width}{\ht0}{$\Bigg\langle$}}
\usebox{0}
\mathclose{\resizebox{1.5\width}{\ht0}{$\Bigg\rangle$}}
\end{equation} has the following properties:\\
\noindent \textup{(i)} The presentation of $\Gamma$ satisfies the $C'(\frac{1}{6})$ small cancellation condition, hence $\Gamma$ is Gromov hyperbolic.\\
\noindent \textup{(ii)} $N=\langle x,y \rangle$ is a normal  subgroup of $\Gamma$ and there exists a short exact sequence $$1\rightarrow N \rightarrow\Gamma \rightarrow Q\rightarrow 1.$$\end{theorem}

\begin{rmk}{\em In Rips' construction the $2$-generated group $N$ can be chosen to be {\em perfect}, i.e. $N=[N,N]$. For example, it is possible to choose the first four relations in (\ref{Rips group}) as follows: \begin{align*}
xy^{}xy^{2}\cdots xy^{p_1}&=1 \\
xy^{p_2}xy^{p_2+1}\cdots xy^{p_2+p_1-1}&=1\\
xy^{p_4}xy^{p_4+1}\cdots xy^{p_4+p_3-1}&=1 \\  
xy^{p_5}xy^{p_5+1}\cdots xy^{p_5+p_3-1}&=1\end{align*} where $p_1,\ldots,p_5\in \mathbb{N}$ are large enough, $\textup{gcd}(p_1(p_2-1),p_3(p_5-p_4))=1$, $p_2>p_1$, $p_4>p_2+p_1$ and $p_5>p_4+p_3$. Note that the commutator $[N,N]$ contains the words $$x^{p_1} y^{r_{p_1 1}}, x^{p_1} y^{r_{p_1 p_2}}, x^{p_3} y^{r_{p_3p_4}},x^{p_3} y^{r_{p_3p_5}}$$ where $r_{pq}:=pq+\frac{1}{2}p(p-1)$. It follows that $y^{p_1(p_2-1)},y^{p_3(p_5-p_4)}\in [N,N]$ and hence $y\in [N,N]$. In particular, $x^{p_1},x^{p_3}\in [N,N]$ and so $x\in[N,N]$.}
\end{rmk}

Interesting examples of finitely generated subgroups of hyperbolic groups are constructed by applying the Rips construction to a finitely presented group $Q$ with certain pathologies. By following this point of view and using Corollary \ref{lemma-semisimple1} we obtain the following corollary.

\begin{corollary} \label{semisimple-nonqie} Let $Q$ be a finitely presented group whose Dehn function grows faster than any fixed iterated of the exponential \textup{(}e.g. $Q$ has unsolvable word problem\textup{)}. Let $\Gamma$ be a $C'(\frac{1}{6})$ small cancellation group provided by Theorem \ref{Rips} such that there exists a short exact sequence $$1\rightarrow N \rightarrow \Gamma \rightarrow Q\rightarrow 1$$ where $N$ is $2$-generated and perfect. Let $P=\Gamma \times_{N} \Gamma$, choose a word length function $|\cdot|_{P}:P\rightarrow \mathbb{N}$ and fix $m \in \mathbb{N}$. There exists an infinite sequence $(w_{n})_{n\in \mathbb{N}}$ of elements in $P$ with the property: for every semisimple representation $\rho:P \rightarrow \mathsf{GL}_d(k)$ we have $$\Big|\Big| \mu\big(\rho(w_{n})\big)\Big|\Big|_{\mathbb{E}} \leq \underbrace{\log \cdots \log}_{m-\textup{times}} \big|w_{n}\big|_{P}$$ for all but finitely many $n \in \mathbb{N}$.\end{corollary}

\begin{proof} Fix a left invariant word metric on $\Gamma$ and equip $\Gamma \times \Gamma$ with the product metric. Since $N=[N,N]$, Corollary \ref{lemma-semisimple1} implies that there exist $C_{\rho},c_{\rho}>0$, depending only on $\rho$, such that $$\Big|\Big|\mu\big(\rho(g)\big)\Big|\Big|_{\mathbb{E}}\leq C_{\rho} \big|g \big|_{\Gamma \times \Gamma}+c_{\rho}$$ for every $g \in P$. By using Proposition \ref{dist}, it follows that the distortion of $P$ in $\Gamma \times \Gamma$ grows at least as the Dehn function of the quotient group $Q=\Gamma/N$ and in particular faster than any iterated exponential $\underbrace{\textup{exp}\circ \cdots \circ  \textup{exp}}_{q-\textup{times}}$. In other words, we may choose a sequence $(w_n)_{n \in \mathbb{N}}$ with $$\big|w_n\big|_{\Gamma \times \Gamma} \leq \underbrace{\log \cdots \log}_{(m+1)-\textup{times}} \big|w_{n}\big|_{P}$$ for all but finitely many $n\in \mathbb{N}$. This finishes the proof of the conclusion.\end{proof}

\section{Proof of Theorem \ref{nonqie1} and Theorem \ref{nonqie2}} \label{proofs2}

We now state and prove a more general version of Theorem \ref{nonqie1} which additionally provides an explicit sequence of elements satisfying  (\ref{intro-ineq1}) for every linear representation. We still denote by $F_m$ the free group on $\big\{x_1,\ldots,x_m\big\}$ and by $k$ a local field.

\begin{theorem} \label{nonqie1gen} Let \hbox{$v=\big[b_1,b_2\big]$} be a basic commutator of weight $r \geq 2$, where $b_1,b_2\in F_m$ are basic commutators with $b_2<b_1$ and $w(b_1)+w(b_2)=r$. Suppose that $\mathcal{A}$ is a finite subset of $[F_m,F_m]$ such that the normal subgroup $\langle \langle \mathcal{A}\rangle \rangle$ of $F_m$ contains the element $v(x_1^n,\ldots,x_m^n)$ for every $n\in \mathbb{N}$. Consider the fiber product $\Delta_{\mathcal{A}}=F_m\times_{\langle \langle \mathcal{A}\rangle \rangle}F_m$ and fix a word length function $|\cdot|_{\mathcal{A}}:\Delta_{\mathcal{A}} \rightarrow \mathbb{N}$. For every representation $\rho:\Delta_{\mathcal{A}} \rightarrow \mathsf{GL}_d(k)$ there exists $C_{\rho}>0$ such that for every $n\in \mathbb{N}$: \begin{equation} \Big|\Big| \mu \big(\rho\big( w(d,n) \big)\big)\Big|\Big|_{\mathbb{E}} \leq C_{\rho}\big|w(d,n)\big|_{\mathcal{A}}^{1-\frac{1}{rd+r+w(b_2)}} \end{equation} where $w(d,n):=\big(1,V_d\big(x_1^n,\ldots, x_m^n\big)\big)\in \langle \langle \mathcal{A}\rangle \rangle$ is defined as in Lemma \ref{estimate2}. \end{theorem}

\begin{proof} By applying Lemma \ref{estimate2} for $q=d$, there exists $C_1>0$ such that \begin{align} \label{nonqie1-eq1} \Big|\big(1,V_{d}(x_1^n,\ldots,x_m^n)\big)\Big|_{\mathcal{A}}^{1-\frac{1}{rd+r+w(b_2)}} \geq C_1 n\end{align} for every $n\in \mathbb{N}$. Theorem \ref{mainthm2} implies that the representation $\rho$ is $(d+1)$-normal. By using (\ref{nonqie1-eq1}) there exist constants $C_2>0$, independent of $n\in \mathbb{N}$, such that for all but finitely many $n \in \mathbb{N}$ we have \begin{align*} \Big|\Big| \mu\big(\rho \big(1,V_{d}(x_1^n,\ldots,x_m^n)\big)\big)\Big|\Big|_{\mathbb{E}}&\leq2^{d+1} C_2 \Big( d\big| v(x_1^n,\ldots,x_m^n)\big|_{F_m}+\big|[v(x_1^n\ldots,x_m^n),b_2(x_1^n,\ldots,x_m^n)]\big|_{F_{m}}\Big)\\ & \leq 2^{d+1} C_2C_{d} n\\ & \leq \frac{2^{d+1} C_2 C_{d}}{C_1}\Big|\big(1,V_{d}\big(x_1^n,\ldots,x_m^n\big)\big)\Big|_{\mathcal{A}}^{1-\frac{1}{rd+r+w(b_2)}}\end{align*} where $C_{d}:=(d+2)|v|_{F_m}+2|b_2|_{F_m}$.\end{proof}
\medskip

The following corollary follows immediately from Theorem \ref{nonqie1gen}. Recall that $\gamma_{r}(F_m)$ denotes the $r$-th term in the lower central series of $F_m$.

\begin{corollary} Let $m,r\geq 2$ be two integers. The fiber product $F_m \times_{\gamma_r(F_m)}F_m$ of $F_m$ with respect to $\gamma_r(F_m)$ does not admit a linear representation over a local field which is a quasi-isometric embedding.\end{corollary}

In particular, by Proposition \ref{comm}, $F_m\times F_m$ contains infinitely many non-commensurable finitely generated subgroups satisfying the conclusion of the previous corollary.
\medskip

We now state and prove a more general version of Theorem \ref{nonqie2}.

\begin{theorem} \label{nonqie2gen} Let \hbox{$v=\big[b_1,b_2\big]$} be a basic commutator of weight $r \geq 2$, where $b_1,b_2\in F_m$ are basic commutators and $b_2<b_1$. Suppose that $\mathcal{B}$ is a finite subset of $[F_m,F_m]$ such that the normal subgroup $\langle \langle \mathcal{B}\rangle \rangle$ of $F_m$ contains the element $v(x_1^n,\ldots,x_m^n)$ for every $n\in \mathbb{N}$. Let $Q_{\mathcal{B}}=\langle x_1,\ldots,x_m \big| \mathcal{B} \big\rangle$ and $\Gamma_{\mathcal{B}}$ be a $C'(\frac{1}{6})$ small cancellation group provided \hbox{by Theorem \ref{Rips} such that} $$1 \rightarrow N_{\mathcal{B}} \rightarrow \Gamma_{\mathcal{B}} \rightarrow Q_{\mathcal{B}} \rightarrow 1$$ is a short exact sequence and $N_{\mathcal{B}}$ is finitely generated. Consider $P_{\mathcal{B}} <\Gamma_{\mathcal{B}} \times \Gamma_{\mathcal{B}}$ the fiber product of $\Gamma_{\mathcal{B}}$ with respect to $N_{\mathcal{B}}$ and fix a word length function $|\cdot|_{\mathcal{B}}:P_{\mathcal{B}} \rightarrow \mathbb{N}$. For every representation $\rho:P_{\mathcal{B}} \rightarrow \mathsf{GL}_d(k)$ there exists $C_{\rho}>0$ such that for every $n\in \mathbb{N}$: \begin{equation} \Big|\Big|\mu \big(\rho\big(\delta(d,n)\big) \big)\Big|\Big|_{\mathbb{E}} \leq C_{\rho}\big|\delta(d,n)\big|_{\mathcal{B}}^{1-\frac{1}{rd+r+w(b_2)}} \end{equation} where $\delta(d,n):=\big(1,V_d\big(\overline{x}_1^n,\ldots, \overline{x}_m^n\big)\big)\in N_{\mathcal{B}}$ is defined as in Lemma \ref{estimate3}.
\end{theorem} 

\begin{proof} First, let us note that the presentation of $\Gamma_{\mathcal{B}}$ (see (\ref{Rips group})) is of the form (\ref{estimate3-pr}). Therefore, by Lemma \ref{estimate3} there exists $C_3>0$ such that \begin{align} \label{nonqie2-eq1} \Big|\big(1,V_{d}\big(\overline{x}_1^n,\ldots,\overline{x}_m^n\big)\big)\Big|_{\mathcal{B}}^{1-\frac{1}{rd+r+w(b_2)}} \geq C_3 n\end{align} for every $n \in \mathbb{N}$. By Theorem \ref{mainthm2}, the representation $\rho$ is $(d+1)$-normal and by using (\ref{nonqie2-eq1}), there exist $C_4>0$, independent of $n$, such that for all but finitely many $n \in \mathbb{N}$ we have $$\Big|\Big| \mu\big(\rho \big(1,V_{d}(\overline{x}_1^n,\ldots,\overline{x}_m^n)\big)\big)\Big|\Big|_{\mathbb{E}} \leq C_4n \leq  \frac{C_4}{C_3}\Big|\big(1,V_{d}\big(\overline{x}_1^n,\ldots,\overline{x}_m^n\big)\big)\Big|_{\mathcal{B}}^{1-\frac{1}{rd+r+w(b_2)}}.$$ The proof of the theorem is complete.\end{proof}
 
The following corollary follows directly from the previous theorem.

\begin{corollary}\label{nil-0} Let $m,r\geq 2$ integers and $\Gamma_{m,r}$ a $C'(\frac{1}{6})$ small cancellation group, provided by Theorem \ref{Rips}, such that $$1 \rightarrow N_{m,r} \rightarrow \Gamma_{m,r} \rightarrow F_m/\gamma_r(F_m) \rightarrow 1$$ is a short exact sequence with $N$ finitely generated. The fiber product $P_{m,r}:=\Gamma_{m,r} \times_{N_{m,r}}\Gamma_{m,r}$ does not admit a linear representation over a local field which is a quasi-isometric embedding.\end{corollary}

Since the free nilpotent group $F_m/\gamma_r(F_m)$ is of finite type, it follows by the the 1-2-3-Theorem in \cite{BBMS} that $P_{m,r}$ is a finitely presented group for every $r,m\in \mathbb{N}$. We close this section with the proof of Theorems \ref{nonqie1} \& \ref{nonqie2}.

\medskip
\noindent {\em Proof of Theorems \ref{nonqie1}} \& {\em\ref{nonqie2}}. Note that for any group $\mathsf{H}$ and $g,h\in \mathsf{H}$, the normal closure $\langle \langle [g,h]\rangle \rangle$ in $\mathsf{H}$ contains the element $[g^n,h^n]$ for every $n\in \mathbb{N}$. Now Theorem \ref{nonqie1} and Theorem \ref{nonqie2} follow immediately by applying Theorem \ref{nonqie1gen} and Theorem \ref{nonqie2gen} for $r=2$ and $v(x_1,\ldots,x_m)=[x_j,x_p]$ respectively. $\qed$

\section{A finitely presented example} \label{Examples}
\par The fiber products provided by Theorem \ref{nonqie1} and Theorem \ref{nonqie1gen} are not finitely presented by a result of Grunewald \cite{Grunewald} (see also \cite{BR}). However, Theorem \ref{nonqie2gen} provides several finitely presented examples. We provide here the a presentation of the simplest example $P_{2,2}$ from Corollary \ref{nil-0}.

\begin{Example}\label{exmp-fp}\normalfont{Let $\mathbb{Z}^2=\big \langle a,b \ \big|[a,b] \big\rangle$ be the free abelian group of rank $2$ and consider the word $$V_r(x,y):=xy^{100r+1}xy^{100r+2}\cdots xy^{100r+100}$$ $r\in \mathbb{N}\cup \{0\}$, on the letters $\{x,y\}$. By applying the Rips construction for the given presentation of $Q=\mathbb{Z}^2$ we obtain the Gromov hyperbolic group $$\Gamma_{2,2}=\sbox0{$a,b,x,y \ \Bigg|\ \begin{matrix}[1.2]
[a,b]V_0(x,y) &axa^{-1}V_1(x,y) &a^{-1}xaV_2(x,y) \\ 
 bxb^{-1}V_3(x,y) &b^{-1}xbV_4(x,y) &aya^{-1}V_5(x,y) \\ 
 a^{-1}yaV_6(x,y) &byb^{-1}V_7(x,y) &b^{-1}ybV_8(x,y)
\end{matrix}$}
\mathopen{\resizebox{1.5\width}{\ht0}{$\Bigg\langle$}}
\usebox{0}
\mathclose{\resizebox{1.5\width}{\ht0}{$\Bigg\rangle$}}$$ where $N_{2,2}=\langle x,y \rangle$ is normal in $\Gamma_{2,2}$ and $1\rightarrow N_{2,2} \rightarrow \Gamma_{2,2} \rightarrow \mathbb{Z}^2 \rightarrow 1$ is exact. The fiber product $P_{2,2}=\Gamma_{2,2} \times_{N_{2,2}} \Gamma_{2,2}$ is finitely presented and a finite presentation of $P_{2,2}$\footnote{The generators $a,b,x_L,y_L, x_R,y_R$ of $P_{2,2}$ correspond to the elements $(a,a),(b,b),(x,1),(y,1),(1,x), (1,y)$ respectively.} is obtained by \cite[Thm. 2.2]{BBMS} as follows: $$P_{2,2}=\Big \langle a,b,x_L,y_L,x_R,y_R\ \big |\ \mathcal{R} \Big \rangle$$ where the set of relations $\mathcal{R}$ consists of the following $21$ words\footnote{Since second homotopy module of the given presentation of $\mathbb{Z}^2$ is trivial, there are no relations of the form $Z_{\sigma}$ given by part (4) of Theorem 2.2. in \cite{BBMS}.}: $$ \begin{matrix}[1.3] [a,b]V_0(x_L,y_L)V_0(x_R,y_R) &[x_L,x_R],\ [x_L,y_R] &[y_L,x_R],\ [y_L,y_R] \\
ax_La^{-1}V_1(x_L,y_L) &ax_Ra^{-1}V_1(x_R,y_R) &a^{-1}x_LaV_2(x_L,y_L) &a^{-1}x_RaV_2(x_R,y_R) &\\
ay_La^{-1}V_5(x_L,y_L) &ay_Ra^{-1}V_5(x_R,y_R) &a^{-1}y_LaV_6(x_L,y_L) &a^{-1}y_RaV_6(x_R,y_R) &\\
bx_Lb^{-1}V_3(x_L,y_L) &bx_Rb^{-1}V_3(x_R,y_R) &b^{-1}x_LbV_4(x_L,y_L) &b^{-1}x_RbV_4(x_R,y_R)&\\
by_Lb^{-1}V_7(x_L,y_L) &by_Rb^{-1}V_7(x_R,y_R) &b^{-1}y_LbV_8(x_L,y_L) &b^{-1}y_RbV_8(x_R,y_R).\\
\end{matrix}$$}\end{Example}

\noindent The group $P_{2,2}$ is linear (since $\Gamma_{2,2}$ is, see \cite{Agol, Wise}) and it does not admit a quasi-isometric embedding into any general linear group over a local field.

\section{Examples from Grothendieck pairs} \label{Gr-pairs} 
In this section we discuss some known constructions of fiber products which also fail to admit quasi-isometric embedding into any general linear group over a local field.
\par For a discrete group $\mathsf{H}$ denote by $\hat{\mathsf{H}}$ its profinite completion. Suppose that $\mathsf{H}$ is residually finite group and $\mathsf{K}$ is a subgroup of $\mathsf{H}$. Following \cite{BR}, we say that the pair $(\mathsf{K},\mathsf{H})$ is a {\em Grothendieck pair} if the inclusion $\iota:\mathsf{K} \xhookrightarrow{}\mathsf{H}$ is not an isomorphism while the induced map $\hat{\iota}:\hat{\mathsf{K}}\rightarrow \hat{\mathsf{H}}$ between profinite completions is.

\subsection{Platonov--Tavgen criterion} Fiber products have been previously used for the construction of Grothendieck pairs. Platonov--Tavgen \cite{PT} constructed the first examples of Grothendieck pairs $(\mathsf{K},\mathsf{H})$ of finitely generated residually finite groups. More precisely, their methods provide the following criterion for the existence of Grothendieck pairs between fiber products, see also \cite[Thm. 6.3]{Bass-Lubotzky} and \cite[Thm. 5.1]{Bridson-Grunewald}.

\begin{theorem}\textup{(Platonov--Tavgen \cite{PT})}  \label{PT-criterion} Let $1\rightarrow N \rightarrow \Gamma \rightarrow Q\rightarrow 1$ be a short exact sequence of groups, where $\Gamma$ is finitely generated, and set $P=\Gamma \times_N \Gamma$. Suppose that $Q$ has no non-trivial finite quotient and $H_2(Q,\mathbb{Z})=0$. Then the inclusion $\iota: P \xhookrightarrow{} \Gamma \times \Gamma$ induces an isomorphism $\hat{\iota}:\hat{P}\rightarrow \hat{\Gamma}\times \hat{\Gamma}$ of profinite groups.\end{theorem}

\begin{Example} \label{PTpair}\normalfont{(Platonov--Tavgen \cite{PT1}) Let $F_4$ be the free group on four generators $\big\{a,b,c,d\big\}$. The {\em Higman group} introduced by Higman in \cite{Higman} is the group with presentation $$H=\Big \langle a,b,c,d \ \big| \ a^{-1}bab^{-2}, b^{-1}cbc^{-2}, c^{-1}dcd^{-2}, d^{-1}ada^{-2} \Big \rangle.$$ The group $H$ does not admit non-trivial finite quotients and $H_2(H,\mathbb{Z})=0$ (e.g. see \cite[Lem. 4.2]{Bridson-Grunewald}). Consider the associated short exact sequence $1\rightarrow N \rightarrow F_4\rightarrow H \rightarrow 1$ and  $P:=F_4\times_N F_4$. Since $H$ is an infinite group the fiber product $P$ is not finitely presented by \cite{Grunewald} and hence cannot be isomorphic to $F_4\times F_4$. Theorem \ref{PT-criterion} \hbox{implies that $(P, F_4\times F_4)$ is a Grothendieck pair.}}\end{Example}

\subsection{Bass--Lubotzky examples} Bass--Lubotzky in \cite{Bass-Lubotzky} constructed the first examples of Gro-\\thendieck pairs $(\mathsf{K},\mathsf{H})$ of finitely generated linear groups which are representation superrigid and $\mathsf{K}$ is not commensurable to a lattice in any product $G_1(k_1)\times \cdots \times G_{m}(k_m)$, where $G_i$ is a simple algebraic group defined over a local field $k_i$. Their examples are constructed as the pair $(\Lambda \times_{L} \Lambda, \Lambda \times \Lambda)$\footnote{Lubotzky in \cite{Lubotzky} exhibited similar examples where $\Lambda$ is a cocompact lattice in $\mathsf{Sp}(d,1)$, $d \geq 2$.}, where $\Lambda$ is a cocompact lattice in the exceptional simple rank $1$ Lie group $\mathsf{F}_{4}^{(-20)}$ and $L$ is a normal infinite index subgroup of $\Lambda$ such that the quotient $\Lambda/L$ satisfies the conditions of the Platonov--Tavgen criterion.

\subsection{Bridson--Grunewald examples} Bridson--Grunewald in \cite{Bridson-Grunewald} used fiber products in order to construct the first examples of Grothendieck pairs of finitely presented residually finite groups. Their examples  provide a negative answer to a question of Grothendieck from \cite{Grothendieck} and are constructed as follows. First, they fix a finitely presented group $Q$ with compact $K(Q,1)$ Eilenberg--MacLane space satisfying the conditions of the Platonov--Tavgen criterion (i.e. $\hat{Q}=1$ and $H_2(Q,\mathbb{Z})=0$). Then they use the construction from \cite{Rips-Wise} to obtain a short exact sequence $$1 \rightarrow N \rightarrow \Gamma \rightarrow Q \rightarrow 1$$ where $N$ is $3$-generated and $\Gamma$ is a residually finite hyperbolic group satisfying the $C'(\frac{1}{6})$ small cancellation condition. The counterexample to Grothendieck's question is the pair of residually finite groups $\big(\Gamma \times_N \Gamma, \Gamma \times \Gamma)$. The fact that the fiber product $\Gamma \times_N \Gamma$ is finitely presented follows by the 1-2-3-Theorem in \cite{BBMS}. 
\medskip

\par Let $R\neq 0$ be a commutative ring. For a group $\mathsf{H}$, denote by $\textup{Rep}_{R}(\mathsf{H})$ the category of representations of $\mathsf{H}$ into $\textup{Aut}_{R}(M)$ where $M$ is a finitely presented $R$-module. Grothendieck established a connection between the profinite completion of a finitely generated group and its category of representations. More precisely, given a group homomorphism $u:\mathsf{K} \rightarrow \mathsf{H}$ of finitely generated groups which induces an isomorphism $\hat{u}:\hat{\mathsf{K}}\rightarrow \hat{\mathsf{H}}$ of profinite groups, Grothendieck proved that the restriction functor $u^{\ast}_{R}:\textup{Rep}_{R}(\mathsf{H})\rightarrow \textup{Rep}_{R}(\mathsf{K})$ is an equivalence of categories (see \cite[Thm. 1.2]{Grothendieck}). In particular, if $(\mathsf{K},\mathsf{H})$ is a Grothendieck pair and $R=k$ is a local field, every linear representation $\rho:\mathsf{K} \rightarrow \mathsf{GL}_d(k)$ extends uniquely to a representation $\hat{\rho}: \mathsf{H} \rightarrow \mathsf{GL}_d(k)$ \hbox{(see the discussion in \cite[\S 3]{Grothendieck}).}

\par In the view of Grothedieck's theorem and the Platonov--Tavgen criterion, one deduces the following bound on the Cartan projection of linear representations of a certain class of fiber products.

\begin{theorem}\label{PTGr} Let $\Gamma$ be a Gromov hyperbolic group and fix a word length function $|\cdot|_{\Gamma}:\Gamma \rightarrow \mathbb{N}$. Suppose that $1\rightarrow N \rightarrow \Gamma\rightarrow Q\rightarrow 1$ is a short exact sequence of groups where $Q$ has no non-trivial finite quotients and $H_2(Q,\mathbb{Z})=0$. For every representation $\rho:\Gamma \times_N \Gamma \rightarrow \mathsf{GL}_d(k)$ there exist $C,c>0$ such that $$\Big| \Big| \mu \big(\rho(\gamma, \gamma w)\big) \Big|\Big|_{\mathbb{E}} \leq C\Big( \big|\gamma\big|_{\Gamma}+\big|\gamma w \big|_{\Gamma}\Big)+c$$ for every $\gamma \in \Gamma$ and $w\in N$. \end{theorem}

\begin{proof} Let us set $P:=\Gamma  \times_N \Gamma$. Note that by Theorem \ref{PT-criterion} the inclusion $\iota: P \xhookrightarrow{} \Gamma \times \Gamma$ induces an isomorphism between profinite completions. By applying Grothendieck's theorem \cite[Thm. 1.2]{Grothendieck} we obtain a representation $\hat{\rho}:\Gamma \times \Gamma \rightarrow \mathsf{GL}_d(k)$ extending $\rho$, i.e. $\hat{\rho} \circ \iota=\rho$. Therefore, since $\Gamma \times \Gamma$ is finitely generated, there exist $C,c>0$ such that $$\Big|\Big|\mu  \big(\rho(\gamma,\gamma,w)\big)\Big|\Big|_{\mathbb{E}}=\Big|\Big|\mu\big(\hat{\rho}(\gamma,\gamma w)\big)\Big|\Big|_{\mathbb{E}} \leq C\Big(\big|\gamma \big|_{\Gamma}+\big|\gamma w\big|_{\Gamma}\Big)+c$$ for every $\gamma\in \Gamma$ and $w\in N$. The conclusion follows. \end{proof}

Theorem \ref{PTGr} immediately applies to the previous fiber product constructions in \cite{PT,Bass-Lubotzky, Bridson-Grunewald}. We also deduce the following corollary.

\begin{corollary} \label{PTGr2} Let $\Gamma$ be a Gromov hyperbolic group and $N$ be a normal subgroup of $\Gamma$. Suppose that $Q=\Gamma/N$ is a finitely presented group which is not Gromov hyperbolic, has no non-trivial finite quotients and $H_2(Q,\mathbb{Z})=0$. Then any representation of $P=\Gamma \times_N \Gamma$ over a local field fails to be a quasi-isometric embedding. Moreover, the group $P$ is not commensurable to any of the fiber produts in Theorem \ref{nonqie1} or Theorem \ref{nonqie2}. \end{corollary}

\begin{proof}  Suppose that there exists a representation $\rho: P \rightarrow \mathsf{GL}_d(k)$ which is a quasi-isometric embedding. By using Theorem \ref{PTGr}, we deduce that the inclusion $\iota: P \xhookrightarrow{} \Gamma \times \Gamma$ is also a quasi-isometric embedding (i.e. the distortion of $P$ in $\Gamma \times \Gamma$ is linear). In particular, by Proposition \ref{dist}, the Dehn function of the quotient $Q=\Gamma/N$ is linear and hence $Q$ has to be Gromov hyperbolic (see \cite{Gromov}). This is a contradiction and hence such representation $\rho$ does not exist.\par \par Now observe that if $P_{0}:=\Gamma_0\times_{L_0} \Gamma_0$ is a fiber product from Theorem \ref{nonqie1} or Theorem \ref{nonqie2}, then any finite extension of a finite-index subgroup of $\Gamma_0/L_0$ surjects onto a non-trivial free abelian group. On the other hand, every finite extension of the group $\Gamma/N$ has finitely many subgroups of finite index. It follows by Proposition \ref{comm} that $P_0$ \hbox{cannot be commensurable to $P$.} \end{proof} In particular, the previous corollary shows that fiber product in Example \ref{PTpair} and the Bridson--Grunewald fiber products in \cite[\S 7, \S8]{Bridson-Grunewald} do not admit quasi-isometric embedding into any general linear group over a local field.

\begin{rmk} \normalfont{Let $\Lambda$ be a superrigid cocompact lattice in either $\mathsf{Sp}(d,1)$, $d \geq 2$, or the exceptional rank $1$ Lie group $\mathsf{F}_4^{(-20)}$. In contrast to Theorems \ref{nonqie1gen} \& \ref{nonqie2gen} and Corollary \ref{PTGr}, by using Corlette's Archimedean superrigidity \cite{Corlette}, we exhibited examples of Zariski dense infinite index subgroups $\mathsf{H}$ of $\Lambda \times \Lambda$ all of whose discrete faithful representations into any real semisimple Lie group are quasi-isometric embeddings. More precisely, the group $\mathsf{H}$ is constructed as the fiber product $\Lambda \times_N \Lambda$, where $\Lambda/N$ is Gromov hyperbolic (see \cite[Thm 1.1 \& Prop. 4.1]{Ts21}).}\end{rmk}

Let us remark that when $(\Gamma \times_N \Gamma, \Gamma \times \Gamma)$ fails to be a Grothendieck pair, then, linear representations of $\Gamma \times_N \Gamma$ might not extend to the product $\Gamma \times \Gamma$. We close this section with such an example where Theorem \ref{nonqie2} applies but Theorem \ref{PTGr} does not. 

\begin{Example}\label{non-extension} \normalfont{Let $F_2$ be the free group on $\{x_1,x_2\}$ and let $\gamma_2(F_2)=[F_2,F_2]$ and $\gamma_3(F_2)=[F_2,\gamma_2(F_2)]$. Observe that $\Delta_2=F_2\times_{\gamma_2(F_2)}F_2$ is a normal subgroup of $F_2 \times F_2$ and $\Delta_3:=F_2\times_{\gamma_3(F_2)}F_2$ is a normal subgroup of $\Delta_2$ with $\Delta_2/\Delta_3 \cong \gamma_2(F_2)/\gamma_3(F_2)\cong \mathbb{Z}$. Let $\pi:\Delta_2 \twoheadrightarrow \mathbb{Z}$ be the surjective homomorphism $$\pi(\delta, \delta w)=w \gamma_3(F_2),\ \delta \in F_2, w \in \gamma_2(F_2)$$ with kernel $\Delta_3$. Consider the representation $\pi_{\lambda}:\Delta_2 \rightarrow \mathsf{SL}_2(k)$ obtained by postcomposing $\pi$ with the homomorphism $\mathbb{Z}=\langle t \rangle \rightarrow \mathsf{SL}_2(k)$, $t \mapsto \textup{diag}(\lambda, \lambda^{-1})$, $\lambda \in k^{\ast}$. Let us set $w_{12}:=(1,[x_1,x_2])$ so that $\pi(w_{12})=t$ and fix a representation $\rho:\Delta_2 \rightarrow \mathsf{SL}_d(k)$. We claim that if $|\lambda|>\max \big\{\ell_1(\rho(w_{12})),\ell_1(\rho(w_{12}^{-1})\big\}$ then the product representation $$\pi_{\lambda}\times \rho\times \rho:\Delta_2 \rightarrow \mathsf{SL}\big(k^2\oplus k^{2d}\big), \ \gamma \mapsto \textup{diag}\big(
\pi_{\lambda}(\gamma), \rho(\gamma), \rho(\gamma)\big)$$ does not extend to a representation $\psi: F_2\times F_2 \rightarrow \mathsf{SL}(k^2\oplus k^{2d})$. \par Suppose that such an extension $\psi$ exists. Let us observe that $$\frac{\ell_1(\psi(\gamma))}{\ell_2(\psi(\gamma))}=\max\Big\{1, \frac{\ell_1(\pi_{\lambda}(\gamma))}{\ell_1(\rho(\gamma))}\Big\}$$ so if $\psi(\gamma)$ is $1$-proximal then $\ell_1(\pi_{\lambda}(\gamma))>\ell_1(\rho(\gamma))$ and its attracting fixed point is in $\mathbb{P}(k^2\oplus 0_{2d})$. By the choice of $|\lambda|>1$ it follows that the proximal limit set of the image of $\pi_{\lambda} \times \rho \times \rho$ in $\mathbb{P}(k^2\oplus k^{2d})$ contains two elements, namely the two lines $[e_1]$ and $[e_2]$ (where $\{e_1,e_2\}$ is the canonical basis of $k^2$). Since $\Delta_2$ is normal in $F_2\times F_2$, $\psi(F_2\times F_2)$ has to preserve the limit set of $\psi|_{\Delta_2}=\pi_{\lambda} \times \rho \times \rho$. In particular, there exists a normal finite index subgroup $L$ of $F_2 \times F_2$ such that $\psi(L)$ fixes the lines $ke_1$ and $ke_2$. Note that there exists $r\in \mathbb{N}$ such that $(1,x_1^r),(1,x_2^r)\in L$ and hence $\psi\big(1,[x_1^r,x_2^r]\big)$ acts trivially on $k^2=ke_1\oplus ke_2$. On the other hand, we have that $$\pi_{\lambda}\big((1,[x_1^r,x_2^r])\big)=\pi_{\lambda}\big((1,[x_1,x_2]^{r^2})\big)=\pi_{\lambda}(w_{12})^{r^2}$$ since $[x_1^r,x_2^r][x_1,x_2]^{-r^2}\in \gamma_3(F_m)$ and $\{1\}\times \gamma_3(F_2) \subset \textup{ker}\pi$. In particular, the restriction of $\psi(1,[x_1^r,x_2^r])=(\pi_{\lambda}\times \rho\times \rho)(1,[x_1^r,x_2^r])$ on the subspace $k^2 \oplus 0_{2d}$ is non-trivial. This is a contradiction, hence an extension $\psi:F_2\times F_2\rightarrow \mathsf{SL}(k^2\oplus k^{2d})$ of $\pi_{\lambda}\times \rho \times \rho$ cannot exist. $\qed$}\end{Example}

\begin{appendix} \section{Distortion of fiber products into direct products and commensurability}\label{Distortion}
\subsection{Distortion of fiber products into products} We review here a fact for the distortion of a fiber product into the ambient direct product.
 \par Suppose that $\Gamma$ is a finitely presented group, $S$ is a finite generating subset of $\Gamma$ and $\mathcal{F}$ is a finite subset of $\Gamma$. Let us set $P:=\Gamma \times_N \Gamma$, where $N=\langle \langle \mathcal{F} \rangle \rangle$, and fix the word length function $|\cdot|_{P}:P \rightarrow \mathbb{N}$ with respect to the finite generating subset $\big\{(g,g):g \in S\big\}\cup \big\{(1,w):w \in \mathcal{F} \big\}$ of $P$. The {\em distortion function} $\textup{Dist}_{P}^{\Gamma \times \Gamma}:\mathbb{N} \rightarrow \mathbb{N}$ of $P$ in $\Gamma \times \Gamma$ is defined as follows: $$\textup{Dist}_{P}^{\Gamma \times \Gamma}(q):=\max \Big\{ \big|(\gamma, \gamma w)\big|_{P}: \big|\gamma\big|_{\Gamma} +\big|\gamma w \big|_{\Gamma}\leq q \Big\}.$$ Observe that a different choice of left invariant word metric on $P$ gives a different distortion function equivalent to $\textup{Dist}_{P}^{\Gamma \times \Gamma}$.
\par Let $F_m$ be the free group on $\{a_1,\ldots,a_m\}$, $|\cdot|_{F_m}$ be the standard word metric on $F_m$ and let us fix a presentation \hbox{$\mathcal{P}_{\Gamma}:=\big \langle a_1,\ldots, a_m| \mathcal{R}\big \rangle$ of $\Gamma$.} If $w\in F_m$ is a word of the form $w=x_{i_1}^{\pm 1}\cdots x_{i_s}^{\pm 1}$, $x_{i_1},\ldots,x_{i_m}\in \{a_1,\ldots,a_m\}$, we declare its length as $\ell(w)=s$. For $w \in \langle \langle \mathcal{R} \rangle \rangle$ define its area $$\textup{Area}(w):=\min \Big\{r: w=\prod_{i=1}^{r} g_i R_i g_{i}^{-1}, g_1,\ldots,g_r\in F_m, R_1,\ldots,R_r\in \mathcal{R}^{\pm1}\Big\}.$$ The Dehn function $\delta_{\mathcal{P}_{\Gamma}}:\mathbb{N}\rightarrow \mathbb{N}$ of the presentation $\mathcal{P}_{\Gamma}$ of $\Gamma$ is defined as follows $$\delta_{\mathcal{P}_{\Gamma}}(q):=\max\Big\{\textup{Area}(w): w \in \langle \langle \mathcal{R} \rangle \rangle, \ell(w)\leq q\Big\}.$$ For two functions $f:\mathbb{N}\rightarrow \mathbb{R}^{+}$ and $g:\mathbb{N}\rightarrow \mathbb{R}^{+}$ we write $f \preceq g$ if there exist constants $c_1,c_2,c_3>0$ such that $g(n)\leq c_1f\big(c_2 n\big)+c_3n$ for every $n \in \mathbb{N}$. Given two isomorphic finite presentations $\mathcal{P}_1$ and $\mathcal{P}_2$, then $\delta_{\mathcal{P}_1} \preceq \delta_{\mathcal{P}_2}$ and $\delta_{\mathcal{P}_2} \preceq \delta_{\mathcal{P}_1}$. We refer the reader to \cite{Ger} and the references therein for more background on Dehn functions.
\par The following proposition provides a lower bound on the distortion of the fiber product $P$ in $\Gamma \times \Gamma$ in terms of the Dehn function of a presentation of the quotient group $\Gamma/N$. It is a consequence of a more general result in \cite[Prop. 3.2]{IT}. I would like to thank Claudio Llosa Isenrich for referring me to \cite[Prop. 3.2]{IT} which implies Proposition \ref{dist}. In the case where $\mathcal{R}$ is empty (i.e. $\Gamma$ is a free group) the following proposition follows by \cite[Thm. 2]{OS}. 

\begin{proposition}\label{dist} Let $\Gamma$ be a finitely presented group and $\mathcal{P}_{\Gamma}=\langle a_1,\ldots,a_m | \mathcal{R} \rangle$ be a finite presentation of $\Gamma$. Let also $\mathcal{F}$ be a finite subset of $\Gamma$, let $N=\langle\langle \mathcal{F}  \rangle\rangle$ and fix a presentation $\mathcal{P}_{\Gamma/N}$ of the quotient group $\Gamma/N$. Then for $q\in \mathbb{N}$ we have $$\delta_{\mathcal{P}_{\Gamma}}\big(\textup{Dist}_{P}^{\Gamma \times \Gamma}(q)\big)\succeq \delta_{\mathcal{P}_{\Gamma/N}}(q).$$ \end{proposition}

For the reader's convenience we provide a proof of the previous proposition.

\begin{proof} Let $\pi:F_m \twoheadrightarrow \Gamma$ be the projection of $F_m$ onto $\Gamma$ and $\mathcal{S}$ be a subset of $F_m$ such that $\pi(\mathcal{S})=\mathcal{F}$. We may assume that $\mathcal{P}_{\Gamma/N}=\langle a_1,\ldots,a_m | \mathcal{R} \cup \mathcal{S} \rangle$ since the Dehn functions of two presentations of the same group are equivalent. We want to obtain an upper bound for the Dehn function $\delta_{\mathcal{P}_{\Gamma/N}}$. Let $q\in \mathbb{N}$ and $w \in \langle \langle \mathcal{R} \cup \mathcal{S} \rangle \rangle$ be an element with $\ell(w)\leq q$. We may find $w_1,\ldots, w_{d+1} \in F_m$ and $s_1,\ldots, s_d \in \mathcal{S}^{\pm 1}$ such that $$(1,w)\big(\langle\langle \mathcal{R} \rangle \rangle \times \langle\langle \mathcal{R} \rangle \rangle\big)=(w_1,w_1)(1,s_1)\cdot \cdot \cdot (w_d,w_d)(1,s_d)(w_{d+1},w_{d+1})\big(\langle\langle \mathcal{R} \rangle \rangle \times \langle\langle \mathcal{R} \rangle \rangle\big),$$ $w_ 1\cdot\cdot \cdot w_{d+1}\in \langle \langle \mathcal{R} \rangle \rangle$ and $|(1,\pi(w))|_{P}=\sum_{i=1}^{d+1} \ell_i+d$, where each $w_i \in F_m$ is a product of $\ell_i$ elements in $ \{a_1^{\pm 1},\ldots,a_m^{\pm 1}\}$. Now observe that $W_d:=\big(w_1s_1\cdot \cdot \cdot w_d s_dw_{d+1}\big)^{-1}w \in \langle \langle \mathcal{R} \rangle \rangle$ and \begin{align*} \big|W_d\big|_{F_m}\leq \sum_{i=1}^{d}|s_i|_{F_m}+\sum_{i=1}^{d+1}\ell_i+q  & \leq \max_{s \in \mathcal{S}^{\pm 1}}|s|_{F_m} \Big(\big|(1, \pi(w)) \big|_{P} +q\Big)\\ & \leq \max_{s \in \mathcal{S}^{\pm 1}}|s|_{F_m}\Big(\textup{Dist}_{P}^{\Gamma \times \Gamma}(q) +q\Big).\end{align*} 

By the definition of the Dehn function $\delta_{\mathcal{P}_{\Gamma}}$, we may write $$W_d=z_1R_iz_1^{-1}\cdots z_{\ell}R_{\ell}z_{\ell}^{-1}$$ where \hbox{$\ell \leq \delta_{\mathcal{P}_{\Gamma}}\big(|W_d|_{F_m}\big)$} and $R_1,\ldots ,R_{\ell}\in  \mathcal{R}^{\pm1}$. Similarly, since the element $w_1\cdots w_{d+1}\in \langle \langle \mathcal{R} \rangle \rangle$ is a product of at most $\big|(1,\pi(w))\big|_{P}$ elements in $\{a_1^{\pm 1},\ldots, a_m^{\pm 1}\}$ and $\big|(1,\pi(w))\big|_{P}\leq \textup{Dist}_{P}^{\Gamma \times \Gamma}(q)$, we may write $$w_1\cdots w_{d+1}=v_1 R_1v_{1}^{-1}\cdots v_s R_s v_{s}^{-1}$$ for some $s\leq \delta_{\mathcal{P}_{\Gamma}}\big(\textup{Dist}_{P}^{\Gamma \times \Gamma}(q)\big)$. Therefore, we can write $$w=\prod_{i=1}^{d} (w_1\cdot\cdot \cdot w_i)s_i (w_1\cdot\cdot \cdot w_i)^{-1} \cdot \prod_{i=1}^{s} v_i R_iv_{i}^{-1}\cdot \prod_{j=1}^{\ell}z_jR_jz_j^{-1}$$ as product of at most $d+s+\ell \preceq  \delta_{\mathcal{P}_{\Gamma}}\big(\textup{Dist}_{P}^{\Gamma \times \Gamma}(q)\big)$ conjugates of elements in $S^{\pm 1}  \cup \mathcal{R}^{\pm 1}$. It follows that the Dehn function of $ \langle a_1,\ldots,a_m| \mathcal{R} \cup \mathcal{S} \rangle$ is dominated by $\delta_{\mathcal{P}_{\Gamma}}\big(\textup{Dist}_{P}^{\Gamma \times \Gamma}(q)\big)$.\end{proof}

\subsection{Commensurable fiber products} We also prove the following standard fact concerning commensurability of fiber products of hyperbolic groups.

\begin{proposition} \label{comm} Let $\Gamma_1$ and $\Gamma_2$ be two non-elementary torsion-free hyperbolic groups and $N_1,N_2$ two normal finite-index subgroups of $\Gamma_1,\Gamma_2$ respectively. Suppose the fiber products $\Gamma_1\times_{N_1}\Gamma_1$ and $\Gamma_2\times_{N_2}\Gamma_2$ are commensurable. There exist groups $G_1$, $G_2$ and $G_3$ with the following properties:\\
\noindent \textup{(i)} $G_1$ is a finite extension of a finite-index subgroup of $\Gamma_1/N_1$.\\
\noindent \textup{(ii)} $G_2$ is a finite extension of $G_1$.\\
\noindent \textup{(iii)} $G_3$ is a finite extension of $\Gamma_2/N_2$ and $G_2$ has finite index in $G_3$.
\end{proposition}

\begin{proof} We may assume that there exists a normal finite-index subgroup $Q_1$ of $\Gamma_1\times_{N_1}\Gamma_1$ which is isomorphic to a finite-index subgroup of $\Gamma_2\times_{N_2}\Gamma_2$. There exist normal finite-index subgroups $\Gamma_1'\subset \Gamma_1$ and $N_1'\subset N_1$ such that $Q_1\cap \textup{diag}(\Gamma_1\times \Gamma_1)=\textup{diag}(\Gamma_1'\times \Gamma_1')$ and $Q\cap \big(\{1\}\times N_1\big)=\{1\}\times N_1'$. Note that since $\Gamma_1'\cap N_1$ has finite index in $N_1$, we may choose a characteristic finite-index subgroup $N_{1}''$ of $N$  contained in $\Gamma_1' \cap N_1$. In particular, $N_{1}''$ is normal in $\Gamma_1'$ and $\Gamma_{1}'\times_{N_{1}''}\Gamma_{1}'$ is a finite index subgroup of $Q_1$. We set $G_1:=\Gamma_1'/N_{1}''$ and observe that $G_1$ maps onto the finite-index subgroup $\Gamma_1'/\Gamma_{1}'\cap N_1=\Gamma_1'N_1/N_1$ of $\Gamma_1/N_1$ with finite kernel, so (i) is satisfied. 
\par By assumption, there exists a monomorphism $\varphi:\Gamma_1'\times_{N_1''}\Gamma_1'\xhookrightarrow{} \Gamma_2\times_{N_2}\Gamma_2$ with finite-index image, say $Q_2$. Observe that $\varphi( \{1\}\times N_1'')$ and $\varphi(N_1''\times \{1\})$ (resp. $\varphi^{-1}(Q_1\cap (\{1\}\times N_2))$ and $\varphi^{-1}(Q_1\cap (\{1\}\times N_2))$)  commute and are non-abelian. By using the fact that centralizers of non-trivial elements in $\Gamma_1$ and $\Gamma_2$ are cyclic, it follows that $\varphi(N_{1}''\times N_{1}'')\subset N_2\times N_2$ and $\varphi^{-1}(Q_2\cap (N_2\times N_2))\subset N_1''\times N_1''$. In particular, $\varphi(N_1''\times N_1'')$ has finite index in $N_2\times N_2$. We may choose a characteristic finite-index subgroup $R_2$ of $N_2\times N_2$ contained in $\varphi(N_1''\times N_1'')$. The group $R_2$ is normal in $\Gamma_2\times_{N_2}\Gamma_2$, hence we set $G_2:=Q_2/R_2$ and $G_3:=(\Gamma_2 \times_{N_2}\Gamma_2)/R_2$. Note that $G_2$ maps onto $Q_2/\varphi(N_{1}''\times N_{1}'')=\Gamma_1'/N_{1}''=G_1$ with finite kernel, so (ii) is satisfied. Moreover, $G_2$ has finite index in $G_3$ which in turn maps onto $\Gamma_2 \times_{N_2}\Gamma_2/N_2\times N_2=\Gamma_2/N_2$ with finite kernel.  \end{proof}
\end{appendix}


\begin{thebibliography}{100}
\bibitem{AMS} H. Abels, G. Margulis and G. Soifer, {\em Semigroups containing proximal linear maps}, Israel J. Math. {\bf 91} (1995), 1--30.


\bibitem{Agol} I. Agol, with an appendix by I. Agol, D. Groves, and J. Manning, {\em The virtual Haken Conjecture}, Doc. Math. {\bf18} (2013), 1045-1087.


\bibitem{Arenas} M. Arenas, {\em A cubical Rips construction}, preprint: \href{https://arxiv.org/abs/2202.01048}{arXiv:2202.01048}, 2022.


\bibitem{Bass-Lubotzky} H. Bass and A. Lubotzky, {\em Nonarithmetic superrigid groups: Counterexamples to Platonov's conjecture}, Annals of Math.  {\bf 151} (2000), 1151-1173.


\bibitem{BBMS} G. Baumslag, M. R. Bridson, C. F. Miller III, and H. Short, {\em Fibre products, nonpositive curvature, and decision problems}, Comment. Math. Helv.  {\bf75} (2000), 457--477.


\bibitem{BMS} G. Baumslag, C. F. Miller III, and H. Short, {\em Isoperimetric inequalities and the homology of groups}, Invent. Math.  {\bf113} (1993), 531--560.


\bibitem{BR} G. Baumslag and J. E. Roseblade, {\em Subgroups of direct products of free groups}, J. London Math. Soc. {\bf 2} (1984), 44-52.


\bibitem{BO} I. Belegradek and D. Osin, {\em Rips construction and Kazhdan property \textup{(T)}}, Groups, Geom. Dyn. {\bf 2} (2008), 1-12.


\bibitem{benoist-limitcone} Y. Benoist, {\em Propri\'et\'es asymptotiques des groupes lin\'eaires}, Geom.  and Funct. Anal.  {\bf 7} (1997), 1--47.

\bibitem{Bridson-Grunewald} M. R. Bridson and F. J. Grunewald, {\em Grothendieck's problems concerning profinite completions and representations of groups}, Annals of Math. {\bf 160} (2004), 359-373.


\bibitem{BH} M. R. Bridson and A. Haefliger, Metric spaces of non-positive curvature, Grundlehren der
Mathematischen Wissenschaften, 319. Springer-Verlag, Berlin, 1999.


\bibitem{BR} M. R. Bridson and A. W. Reid, {\em Nilpotent completions of groups, Grothendieck pairs, and four problems of Baumslag}, Int. Math. Res. Not. {\bf 8} (2015), 2111-2140.



\bibitem{Bri} M. R. Bridson, {\em The strong profinite genus of a finitely presented group can be infinite}, J. Eur. Math. Soc. {\bf 18} (2016), 1909-1918.


\bibitem{BT} F. Bruhat and J. Tits, {\em Groupes r\'eductifs sur un corps local: I. Donn\'ees radicielles
valu\'ees,} Publ. Math. Inst. Hautes \'Etudes Sci. {\bf 41} (1972), 5--251.


\bibitem{CMZ} A. E. Clement, S. Majewicz S. and M. Zyman, The theory of nilpotent groups, Vol. 43. Springer International Publishing, 2017.

\bibitem{CDP} M. Coornaert, T. Delzant, and A. Papadopoulos, G\'eom\'etrie et th\'eorie des groupes: les groupes hyperboliques de Gromov, Vol. 1441. Springer, 2006.


\bibitem{Corlette} K. Corlette, {\em Archimedean superrigidity and hyperbolic geometry}, Annals of Math. \textbf{135} (1992), 165--182. 



\bibitem{DGK} J. Danciger, F. Gu\'eritaud, and F. Kassel, {\em Convex cocompact actions in real projective geometry}, preprint, \href{https://arxiv.org/abs/1704.08711}{arXiv:1704.08711}, 2017.


\bibitem{DFWZ} S. Douba, B. Flechelles, T. Weisman and F. Zhu, {\em Cubulated hyperbolic groups admit Anosov representations}, preprint: \href{https://arxiv.org/abs/2309.03695}{arXiv:2309.03695}, 2023.
\bibitem{Ger} S. M. Gersten, {\em Isoperimetric and isodiametric functions of finite presentations}, Geometric group theory 1 (1993), 79-96.


\bibitem{Gromov} M. Gromov, Hyperbolic groups, in {\em Essays in Group Theory}, Ed. M. Gersten, MSRI publications, p. 75-263 Springer Verlag, 1987.


\bibitem{Grothendieck} A. Grothendieck, {\em Repr\'esentations lin\'eaires et compactification profinie des groupes discrets}, Manuscripta Math. {\bf 2} (1970), 375--396.


\bibitem{Grunewald} F. J. Grunewald, {\em On some groups which cannot be finitely presented}, J. London Math. Soc. {\bf 17} (1978), 427--436.


\bibitem{GGKW} F. Gu\'eritaud, O. Guichard, F. Kassel and A. Wienhard, {\em Anosov representations and proper actions}, Geom. Top. {\bf 21} (2017), 485--584.


\bibitem{GW} O. Guichard and A. Wienhard, {\em Anosov representations: Domains of discontinuity and
applications},  Invent. Math. {\bf 190} (2012), 357--438.


\bibitem{HW} Haglund F. and D.T. Wise, {\em Special cube complexes,} Geom. Funct. Anal. {\bf 17} (2008), 1551-1620.


\bibitem{Hidber} C. Hidber, {\em Isoperimetric functions of finitely generated nilpotent groups}, J. Pure Appl. Algebra {\bf 144} (1999), 229--242.


\bibitem{Higman} G. Higman, {\em A finitely generated infinite simple group}, J. London Math. Soc. {\bf 26} (1951), 61--64.


\bibitem{IT} C. Llosa Isenrich and R. Tessera, {\em Residually free groups do not admit a uniform polynomial isoperimetric function}, Proc. Am. Math. Soc. {\bf 148}, 4203--4212.


\bibitem{Labourie} F. Labourie,  {\em Anosov flows, surface groups and curves in projective space},  Invent. Math. {\bf 165} (2006), 51--114.


\bibitem{Burnside}V. Lomonosov and P. Rosenthal, {\em The simplest proof of Burnside's theorem on matrix algebras}, Linear algebra and its applications {\bf 383} (2004), 45-47.



\bibitem{Lor} F. Lorenz, Algebra: Volume ii: Fields with structure, algebras and advanced topics. Springer Science \& Business Media, 2007.


\bibitem{LMR} A. Lubotzky, S. Mozes and M. S. Raghunathan, The word and Riemannian metrics on lattices of semisimple groups, {\em Publ. Math. Inst. Hautes \'Etudes Sci.} {\bf91} (2000), 5--53.


\bibitem{Lubotzky} A. Lubotzky, {\em Some more Non-arithmetic Rigid groups}, Geometry, Graphs and Dynamics: Proceedings in Memory of Robert Brooks, Israel Mathematical Conference Proceeding (IMCP), Contemporary Mathematics  {\bf 387} (2005), 237--244, AMS, Providence, RI.



\bibitem{LS} R. C. Lyndon and P. E. Schupp, {\em Combinatorial Group Theory}, Vol. 188. Berlin: Springer, 1977.



\bibitem{Magnus} W. Magnus, {\em Beziehungen zwischen Gruppen und Idealen in einem speziellen Ring}, Math. Ann. {\bf 111} (1935), 259--280.

\bibitem{Mihailova} K. A. Mihailova, {\em The occurrence problem for free products of groups}, Math. USSR Sb. {\bf 4}, p.181.


\bibitem{OS} A. Yu Olshanskii and M. V. Sapir, {\em Length and Area Functions on Groups and Quasi-Isometric Higman Embeddings}, Int. J. Alg. Comp. {\bf 11} (2001), 137--170.


\bibitem{PT} V. P. Platonov and O. I. Tavgen, {\em Grothendieck's problem on profinite completions of groups}, Soviet Math. Doklady {\bf 33} (1986), 822-825.


\bibitem{PT1} V. P. Platonov and O. I. Tavgen, {\em Grothendieck's problem on profinite completions and representations of groups}, K-theory {\bf 4} (1990), 89-101.


\bibitem{Rips} E. Rips,  {\em Subgroups of small cancellation groups}, Bull. London Math Soc. {\bf 14} (1982), 45--47.



\bibitem{Ts21} K. Tsouvalas, {\em Fiber products of rank $1$ superrigid lattices and quasi-isometric embeddings}, preprint: \href{https://arxiv.org/abs/2111.11586}{arXiv:2111.11586}, 2021.


\bibitem{Rips-Wise}  D. T. Wise, {\em A residually finite version of Rips's construction}, Bull. London Math Soc. {\bf 35} (2003), 23--29.



\bibitem{Wise} D. T. Wise, {\em Cubulating small cancellation groups}, Geom. Funct. Anal. {\bf 14} (2004), 150-214.


\end{thebibliography}
\end{document}